\documentclass[10pt]{amsart}
\usepackage{etex}
\usepackage{amssymb,eucal}
\usepackage{amscd}
\usepackage[all,cmtip]{xy}
\usepackage{rotating}
\usepackage{hyperref,mathrsfs,soul}
\usepackage{upgreek}
\usepackage{tcolorbox,mathtools}
\usepackage{mathabx,amsmath}

\usepackage{amsmath,amssymb,caption}

\setul{}{1.3pt}

\newtheorem{theorem}{Theorem }[section]

\newtheorem{lemma}[theorem]{Lemma}

\newtheorem*{remark}{Remark}
\newtheorem{corollary}[theorem]{Corollary}
\newtheorem{proposition}[theorem]{Proposition}
\newtheorem{conjecture}[theorem]{Conjecture}

\def\soc{\mathop{\mathrm {soc}}\nolimits}
\def\PG{\mathsf{PG}}

\def\eop{\hspace*{\fill}{\footnotesize$\blacksquare$}}
\newcommand{\Aut}{\mathrm{Aut}}
\newcommand{\id}{\mathrm{id}}

\newcommand{\cB}{\mathcal{B}}
\newcommand{\cC}{\mathcal{C}}

\newcommand{\cD}{\mathcal{D}}

\newcommand{\cL}{\mathcal{L}}

\newcommand{\mP}{\mathcal{P}}

\newcommand{\mS}{\mathcal{S}}

\newcommand{\cX}{\mathcal{X}}

\newcommand{\mL}{\mathcal{L}}
\newcommand{\mW}{\mathcal{W}}
\newcommand{\F}{\mathbb{F}}
\newcommand{\mH}{\mathcal{H}}

\newcommand{\wE}{\widetilde{E}}

\newcommand{\mQ}{\mathcal{Q}}

\newcommand{\PSL}{\mathsf{PSL}}
\newcommand{\PGL}{\mathsf{PGL}}
\newcommand{\PGU}{\mathsf{PGU}}
\newcommand{\PSU}{\mathsf{PSU}}
\newcommand{\PSp}{\mathsf{PSp}}
\newcommand{\Sp}{\mathsf{Sp}}
\newcommand{\SL}{\mathsf{SL}}
\newcommand{\GL}{\mathsf{GL}}
\newcommand{\GU}{\mathsf{GU}}
\newcommand{\SU}{\mathsf{SU}}

\newcommand{\Out}{\mathsf{Out}}

\newcommand{\diag}{\mathsf{diag}}

\newcommand{\la}{\langle}
\newcommand{\ra}{\rangle}
\newcommand{\cS}{\mathcal{S}}
\newcommand{\cP}{\mathcal{P}}

\DeclareFontFamily{U}{matha}{\hyphenchar\font45}
\DeclareFontShape{U}{matha}{m}{n}{
      <5> <6> <7> <8> <9> <10> gen * matha
      <10.95> matha10 <12> <14.4> <17.28> <20.74> <24.88> matha12
      }{}
\DeclareSymbolFont{matha}{U}{matha}{m}{n}
\DeclareFontSubstitution{U}{matha}{m}{n}

\DeclareMathSymbol{\flatcap}           {2}{matha}{"58}
\DeclareMathSymbol{\flatcup}           {2}{matha}{"59}

\usepackage{tikz}
\usepackage{graphicx}
\usepackage[all,cmtip]{xy}
\usepackage{rotating,multicol}
\usetikzlibrary{decorations.markings}
\usetikzlibrary{matrix}
\usetikzlibrary{arrows}
\usepackage{pgf}
\usepackage{lscape}
\usepackage{listings}
\usetikzlibrary{positioning}
\usepackage{mdframed}

\title[Ealy's conjecture]{Ealy's conjecture in odd characteristic}


\author{Tao Feng}
\address{Zhejiang University, School of Mathematical Sciences, 866 Yuhangtang road, Hangzhou, China}

\email{tfeng@zju.edu.cn}

\author{Koen Thas}

\address{{Ghent University},
{Department of Mathematics, Informatics and Statistics},
{Krijgslaan 281, S9, B-9000 Ghent, Belgium}}
\email{koen.thas@gmail.com}

\begin{document}
\maketitle

\begin{abstract}
We solve Ealy's conjecture from 1977 by showing that for each odd prime $p$, a finite generalized quadrangle each point of which admits a central symmetry of order $p$, is either a classical symplectic quadrangle in dimension $3$, or a Hermitian quadrangle in dimension $3$ or $4$. As a byproduct, we vastly generalize the aforementioned result by determining the finite generalized quadrangles whose every point admits at least one nontrivial central symmetry.
 \end{abstract}

\medskip
\begin{tcolorbox}
\setcounter{tocdepth}{1}
\tableofcontents
\end{tcolorbox}

\section{Introduction}
A finite generalized quadrangle  is a spherical building of type $\mathsf{B}_2$. One of the outstanding combinatorial problems in low rank incidence geometry is the elegant but for the moment totally intractable conjecture which predicts that a finite generalized quadrangle whose every dual hyperbolic lines are maximal, can be identified with the $\F_q$-rational points and lines of a nonsingular orthogonal quadric in $\PG(4,q)$ or $\PG(5,q)$ for some prime power $q$. If the parameters of the generalized quadrangle are of type $(a,a)$, the result is actually old and well known, cf. \cite[section 5.2.1]{FGQ} but if this is not the case, no upshot results are known without extra assumptions on automorphism groups. The natural group-theoretical analogue says (through invoking the appropriate Moufang conditions) that all finite generalized quadrangles whose every line is an {\em axis of symmetry}, are isomorphic to one of the aforementioned orthogonal quadrangles. Generalized quadrangles that satisfy this Moufang condition automatically have the property that all dual hyperbolic lines are maximal, but the extra group-theoretical knowledge one now disposes of yields a lot of algebraic machinery which one can use to understand the underlying geometry at a deeper level. Still, it took more than 40 years to develop a {\em geometric} approach to solve the group-theoretical version of the problem, cf. \cite{ACFGQ, SFGQ, JATMOUF}, and in particular \cite{ACFGQ} for more details on the history of this problem. Upon using deeper group-theoretical results, earlier approaches were known that mostly lacked the geometric information the more recent approaches came with. Fong and Seitz famously classified finite groups with a split BN-pair of rank $2$ by showing that such groups are extensions of rank $2$ adjoint Chevalley groups \cite{FS1, FS2} --- a result which is foundational in the first generation proof of the Classification of Finite Simple Groups (CFSG) --- and in geometrical terms, this equates to a complete classification of finite Moufang polygons. The quadrangular case, which is the most difficult one, indeed comprises the finite quadrangles all of whose lines are axes of symmetry. Later, Tent and Van Maldeghem \cite{KTHVM} showed that also in the infinite case Moufang polygons and spilt BN-pairs of rank $2$ are two guises of the same mathematical story. Around about the same time that \cite{FS1,FS2} was published, Tits produced a preprint \cite{JTQM} in which he classified (not necessarily finite) generalized quadrangles each line of which is an axis of symmetry, and each point of which is a center of symmetry. In the finite case, such quadrangles are of type $(a,a)$, and since all hyperbolic lines and dual hyperbolic lines are maximal, a classification was already known. Much later, the second author obtained a Lenz-Barlotti classification of finite generalized quadrangles based on the possible configurations of axes of symmetry \cite{SFGQ}, in which understanding the local behavior of quadrangles with symmetry was decisive. The question remains how much the {\em global} results can be stretched. In \cite{KTirred} for example, it was shown that if a finite generalized quadrangle has a strongly irreducible automorphism group, at least one axis of symmetry is enough to conclude that the quadrangle arises as the point-line geometry of an orthogonal quadric.

In the 1970s and 1980s, Ealy \cite{Ealy} and Walker \cite{walker77} developed much more general variations of the aforementioned global results through deep group-theoretical means. They imposed extra arithmetic properties on the size of certain automorphism groups in their respective work. In \cite{walker77} Walker classified geometrically irreducible finite generalized quadrangles which have a line with at least two nontrivial axial symmetries and a point with at least two nontrivial central symmetries, by showing that such a quadrangle $\Upgamma$ is isomorphic to a symplectic quadrangle $\mW(2^a)$ for some positive integer $a$, and moreover obtaining that the group generated by the axial and central symmetries of $\Upgamma$ must contain the simple group $\mathsf{PSp}_4(2^a)$. Walker also considered other finite buildings of rank $2$ in \cite{walker82,walker83}, but he turned to extra assumptions in order to deduce the appropriate classification results. Note the great generality of Walker's result compared to Tits's approach in the general case!

Ealy vastly generalized Walker's seminal result of \cite{walker77} by completely classifying finite generalized quadrangles in which each point admits at least one central symmetry of order $2$ | and hence only asking ``half" of the assumptions of Walker |  by showing that such quadrangles are isomorphic to either a symplectic quadrangle $\mW(2^a)$, a Hermitian quadrangle $\mH(3,2^{2a})$ or a Hermitian quadrangle $\mH(4,2^{2a})$ for some positive integer $a$. This confirmed a conjecture of Shult \cite[section 9.8]{FGQ}  As in Walker \cite{walker77}, Ealy also obtained accompanying group-theoretical counterparts for simple groups.  Ealy's result is truly remarkable: in case of Moufang quadrangles, one assumes that every root {\em and} dual root gives rise to a maximal group of (dual) root-elations, whereas Ealy only starts with (less than) the bare necessities: one single nontrivial involutory point-symmetry for every point is sufficient to obtain the whole ball of wax!  The innovative  proof comprises Ealy's entire 1977 PhD thesis \cite{Ealy}, and uses a blend of deep group-theoretical results involving involutions, such as Aschbacher's classification of finite groups generated by odd transpositions \cite{oddTrans}. Unfortunately no such results are available in the ``odd" case, and consequently the odd case is still open after almost 50 years. Ealy conjectured in \cite{Ealy} that for each odd prime $p$, a finite generalized quadrangle in which every point admits a group of central symmetries the size of which is divisible by $p$, is classical and arises from a rank $2$ polar space of symplectic or Hermitian type. (So in the odd case, for each odd prime $p$ we have a corresponding Ealy conjecture.) 

Solving Ealy's grand conjecture for each prime is the first main goal of the present paper:\\

\texttt{First Main Result.}\quad
Let $\mS$ be a finite thick generalized quadrangle, and let $r$ be any prime. If the order of the full group of symmetries about each point is divisible by $r$, then $\mS$ is isomorphic to one of $\mW(3,q) $, $\mH(3,q^2)$ and $\mH(4,q^2)$, where $q$ is a power of $r$.\\

In fact, we will do much more: after obtaining Ealy's classical  conjecture, we will classify all finite generalized quadrangles each point of which admits {\em at least one} nontrivial central symmetry: \\

\texttt{Second Main Result.}\quad
Let $\mS$ be a finite thick generalized quadrangle. If each point is the center of a nontrivial symmetry, then $\mS$ is isomorphic to one of $\mW(3,q) $, $\mH(3,q^2)$ and $\mH(4,q^2)$, where $q$ is a prime  power.\\

Our approach contains both group-theoretical, incidence-geometric techniques, and we need to invoke the Classification of Finite Simple Groups at various points. To our great surprise, the reduction of the ``general Ealy problem" to our first main result only requires purely synthetical geometric arguments, and no deep group theory is needed whatsoever. Our main results have applications in the study of locally primitive generalized quadrangles, which we will discuss in the last section.

\medskip
{\bf The present paper}.\quad
In Section \ref{secnot}, we fix the notation for the present paper, and describe Ealy's conjecture in detail. In Section \ref{subsec_prel}, we make a number of preliminary deductions about the structure of the automorphism group of the quadrangle $\cS$ (or rather, the subgroup $G$ generated by the point-symmetries), and show that its socle is a simple group of Lie type.
In Section \ref{secLie} and Section \ref{secclass} --- the longest and most technical sections of this paper --- we handle the cases in which the socle of $G$ is isomorphic to an exceptional group of Lie type, or a classical group. We exclude the exceptional cases through analyzing the possible pairs $(G,G_p)$, where $G_p$ is the stabilizer of a point $p$ in $G$ (which is shown to be a local maximal subgroup of $G$), using a number of arithmetical properties and estimates of both the groups and the underlying geometry. Once the exceptional cases are excluded, we show that if the socle of $G$ is a classical group, the possible groups give rise to the generalized quadrangles in the First Main Result.  (If $G$ is unitary, symplectic or orthogonal, we use detailed knowledge of the associated nondegenerate $\Upomega$-invariant Hermitian, alternating or quadratic form defined on the corresponding natural module $V$ to obtain the desired result, where $\Upomega$ is the quasisimple covering group of the socle of $G$ in $\GL(V)$.) 
This finishes the proof of Ealy's conjecture. In Section \ref{secgen}, we generalize Ealy's conjecture and prove our Second Main Result. The nature of the proof could not be more different than the proof of the First Main Result: we do {\em not} invoke CFSG, nor any other advanced group-theoretical machinery. Rather, we realize a number of geometric properties which eventually allow us to obtain that $G$ acts transitively on the point set of the generalized quadrangle, and that is sufficient to reduce it to the First Main Result. Finally, in Section \ref{sec_Kanto} we consider an application of our main results to the study of locally primitive generalized quadrangles and describe an approach to Kantor's conjecture for such quadrangles by using the Wielandt-Thompson type theorems in \cite{vanBon}. In an appendix, we recall and prove some (unpublished) results of Ealy's thesis for the convenience of the reader.

\medskip
{\bf Acknowledgment}.\quad
Part of the research presented in this paper was carried out while the second author visited Zhejiang University, whose hospitality he gratefully acknowledges.   \\

\medskip
\section{Description of the problem and notation}
\label{secnot}

Suppose that $\cS=(\cP,\cL)$ is a generalized quadrangle of order $(s,t)$, where $s,t>1$. For a point $u\in\cP$, we write $u^\perp=\{x\in\cP:x\sim u\}$ and note that $u \sim u$.  For two distinct points $u,v$, we define $\{u,v\}^\perp=u^\perp\cap v^\perp$ and define $\{u,v\}^{\perp\!\perp}$ to be the set of points that are collinear with all points of $\{u,v\}^\perp$.
The set $\{ u, v \}^{\perp\!\perp}$ is called a {\em hyperbolic line} if $u$ is not collinear with $v$, and if dually the lines $A$ and $B$ are nonconcurrent, then $\{ A, B \}^{\perp\!\perp}$ is a {\em dual hyperbolic line}.
\begin{lemma}\label{lem_1psbound}
If $\cS=(\cP,\cL)$ is a generalized quadrangle of order $(s,t)$, then
\[
  |\cP|^{1/4}<1+s<|\cP|^{2/5}.
\]
\end{lemma}
\begin{proof}
Write $v=|\cP|$, so that $v=(1+s)(1+st)$. We have $s\le t^2$ and $t\le s^2$ by \cite{slet2}. Since $t\le s^2$, we have $1+st<(1+s)^{3}$ which gives the bound $1+s>v^{1/4}$. Since $s\le t^2$, we have $s\le (st)^{2/3}$ and so $v<(1+st)^{5/3}$. It follows that $1+st>v^{3/5}$, and correspondingly $1+s<v^{2/5}$.  This completes the proof.
\end{proof}

Let $\cP',\cL'$ be subsets of $\cP,\cL$ respectively. We say that $(\cP',\cL')$ is a {\em substructure} of $\cS$, if the following properties hold: (i) if $x,y$ are two collinear points in $\cP'$, then the line they determine is in $\cL'$,  (ii) if $\ell,\ell'$ are two concurrent lines in $\cL'$, then their intersection is in $\cP'$, (iii) if $(x,\ell)$ is a non-incident point-line pair, then the unique line through $x$ that intersects $\ell$ is in $\cL'$. By \cite[Theorem~2.1]{walker77}, a substructure is one of the following types:
\begin{enumerate}
  \item[(A)]$\cL'=\emptyset$, and $\cP'$ is a set of pairwise noncollinear points,
  \item[(B)]$\cP'=\emptyset$, and $\cL'$ is a set of pairwise noncoccurent lines,
  \item[(C)]$\cP'$ contains a point $p$ such that $\cP'\subseteq p^\perp$ and all lines of $\cL'$ are incident with $p$,
  \item[(C${}^\textup{d}$)]$\cL'$ contains a line $\ell$ such that $\cL'\subseteq \ell^\perp$ and all points of $\cP'$ are incident with $\ell$,
  \item[(D)]a grid,
  \item[(D${}^\textup{d}$)]a dual grid,
  \item[(E)] a subquadrangle.
\end{enumerate}
An {\em ovoid} of the GQ $\cS$ is a set of $1+st$ points that are pairwise noncollinear.
\begin{proposition}[\cite{FGQ}, 2.2.2 (dual version)] \label{prop_fgq222}
Suppose that $\cS'$ is a subquadrangle of order $(s',t)$ in the thick generalized quadrangle $\cS$ of order $(s,t)$, $s'<s$. Then we have the following:
\begin{itemize}
\item[{\rm (1)}] $s\ge t$, and $s=t$ implies $s'=1$.
\item[{\rm (2)}] If $s'>1$, then $\sqrt{s}\le t'\le s$ and $t^{3/2}\le s\le t^2$.
\item[{\rm (3)}] If $s=t^{3/2}$ and $s'>1$, then $ s'=\sqrt{t}$.
\item[{\rm (4)}] If $\cS'$ contains a proper subGQ $\cS''$ of order $(s'',t)$, then $s''=1$, $s'=t$ and $s=t^2$.
\end{itemize}
\end{proposition}

If $\uptheta$ is a collineation of $\cS$, then we write $\cP_\uptheta$ and $\cL_\uptheta$ for the set of fixed points and fixed lines of $\uptheta$ respectively. We say that $\cS_\uptheta=(\cP_\uptheta,\cL_\uptheta)$ is the fixed substructure of $\uptheta$. A {\em symmetry} $\uptheta$ about a point $p$ is an automorphism of $\cS$ that fixes all points collinear with $p$, and we call $p$ the {\em center} of $\uptheta$. We refer to a symmetry about a point as a {\em central symmetry}, or a {\em point-symmetry}.\medskip

An automorphism group $G$ of a generalized quadrangle $\cS$ is {\em geometrically irreducible} if the only substructure it leaves invariant is $(\emptyset,\emptyset)$ and $\cS$.

\begin{theorem}[Walker \cite{walker77}]
\label{thm_Walker77}
Let $G$ be an automorphism group of the generalized quadrangle $\cS$.
\begin{enumerate}
  \item Suppose that $G$ is geometrically irreducible. If $G$ contains at least $2$ nontrivial symmetries about some point $p$ and at least $2$ nontrivial symmetries about some line $\ell$, then $\cS$ is isomorphic to $\mW(3,q)$ and $G$ contains $\PSp_4(q)$, where $q=2^n$ for some $n\ge 2$.
  \item Suppose that $G$ leaves no point or line of $\cS$ fixed.  If there is a point $p$ and a line $\ell$ such that $G$ contains at least $2$ nontrivial symmetries with center $p$ and at least $2$ nontrivial symmetries with axis $\ell$, then $\cS$ contains a $G$-invariant subquadrangle $\cS'=\mW(3,q)$ and the restriction of $G$ to $\cS'$ contains $\PSp_4(q)$, where $q=2^n$ for some $n\ge 2$.
\end{enumerate}
\end{theorem}

The {\em socle} of a group is the subgroup generated by all its minimal normal subgroups.

\begin{theorem}\label{thm_walker4.1}
Let $G$ be a geometrically irreducible automorphism group of the thick generalized quadrangle $\cS=(\cP,\cL)$. Assume that the full group $\widetilde{E}(p)$ of symmetries at each point is nontrivial, and let $S=\la \widetilde{E}(p):p\in\cP\ra$. Then  $S,G$ are almost simple groups with the same socle.
\end{theorem}
\begin{proof}
This is a special case of the point-line dual version of \cite[Theorem~4.1]{walker77}. Under the assumption of geometrical irreducibility, the cases (2) and (3) in \cite[Theorem~4.1]{walker77} do not occur. 
\end{proof}

For an integer $n$ and a prime $r$, we use $n_{r'}$ for the largest divisor of $n$ that is relatively prime to $r$, and use $n_r$ for the highest power of $r$ that divides $n$. We use the same notation for group structures as in \cite{Atlas}. For two groups $A$ and $B$, $A:B$ is the split extension of $A$ by $B$, $A.B$ is an extension of $A$ by $B$ which may be split or non-split, and $A\circ B$ is a central product of $A$ and $B$.

\subsection{The known classical GQs with a nontrivial symmetry}\label{sec_examples}

Let $\mS$ be a thick generalized quadrangle of order $(s,t)$. For a point $x$, let $\widetilde{E}(x)$ be the full group of symmetries with center $x$. It is known that if there is a nontrivial symmetry about a point in a thick generalized quadrangle, then $s+t\mid st(t+1)$ by \cite[8.1.2]{FGQ}. For two noncollinear points $x,y$, we have the following inequality:
\begin{equation}\label{symineq}
\vert \{ x, y \}^{\perp\!\perp} \vert \ \geq\ \vert \widetilde{E}(x) \vert + 1.
\end{equation}
This follows from the fact that $\widetilde{E}(x)$ acts freely on the points of $\mS$.

Below, we explicitly describe the nontrivial central symmetries in each of the classical generalized quadrangles of order $(s,t)$. Let $\widetilde{E}(p)$ be the full group of symmetries about the point $p$. Since the full automorphism groups of those quadrangles act transitively on points, it suffices to describe $\widetilde{E}(p)$ for a particular point $p$.  Let $V=\F^n$ be a vector space equipped with an alternating form $\upkappa$, a quadratic form $Q$ or a Hermitian  form $H$. Here, $\F=\F_{q^2}$ if the form is unitary and $\F=\F_q$ otherwise. Let $e_1= (1,0,\ldots,0),\ldots,e_n=(0,\ldots,0,1)$ be the standard basis of of $V$. 

\begin{itemize}
  \item[(a)]$\mW(3,q) $, $\upkappa(x,y)=x_1y_2-x_2y_1+x_3y_4-x_4y_3$. We have $(s,t)=(q,q)$, and $\widetilde{E}(\la e_1\ra)=\{t_a:a\in\F_q\}$, where $t_a(x)=x+a\upkappa(x,e_1)e_1$ for $x\in V$.
  \item[(b)]$\mQ(4,q) $,  $Q(x)=x_1x_2+x_3x_4+x_5^2$, $q$ odd, and $(s,t)=(q,q)$. It is well known that for every pair $(x,y)$ of noncollinear points in $\mQ(4,q)$ with $q$ odd, we have that $\vert \{ x, y \}^{\perp\!\perp} \vert = 2$, so by  \eqref{symineq} there is no nontrivial central symmetry.
  \item[(c)]$\mQ^-(5,q) $,  $Q(x)=x_1x_2+x_3x_4+x_5^2+ax_5x_6+bx_6^2$, where $X^2+aX+b$ is irreducible. We have $(s,t)=(q,q^2)$, and $s+t\mid st(t+1)$ does not hold. It follows that there is no nontrivial symmetry about $\la e_1\ra$.
  \item[(d)]$\mH(3,q^2)$, $H(x)=x_1x_2^q+x_2x_1^q+x_3x_4^q+x_4x_3^q$. We have $(s,t)=(q^2,q)$, and it is the point-line dual of $\mQ^-(5,q)$. The group $\widetilde{E}(\la e_1\ra)$ consists of
      \[
        (e_1,\ldots,e_4)\mapsto (e_1,e_2+ae_1,e_3,e_4), \mbox{ where }\ a+a^q=0.
      \]
  \item[(e)]$\mH(4,q^2)$, $H(x)=x_1x_2^q+x_2x_1^q+x_3x_4^q+x_4x_3^q+x_5^{q+1}$. The group $\widetilde{E}(\la e_1\ra)$ has order $q$, which is similarly described as in the $\mH(3,q^2)$ case. Its point-line dual has parameters $(q^3,q^2)$, and $s+t\mid st(t+1)$ does not hold, so the dual quadrangle does not have a nontrivial symmetry.
\end{itemize}

To summarize, a classical generalized quadrangle has a nontrivial central symmetry if and only if it is one of $\mW(3,q) $, $\mH(3,q^2)$ and $\mH(4,q^2)$ for a prime power $q$. The set $\{\la e_1\ra,\la e_2\ra\}^{\perp\!\perp}$ has size $q+1$, and coincides with the set of singular or isotropic points on $\la e_1,e_2\ra_{\F}$ in all those quadrangles. We have $\la \widetilde{E}(e_1), \widetilde{E}(e_2)\ra=\PSL_2(q)$, and it acts $2$-transitively on the above set. Moreover, $\la \widetilde{E}(e_1), \widetilde{E}(e_2)\ra_{\la e_1\ra}$ has $\widetilde{E}(e_1)$ as its unique minimal normal subgroup.

\section{Outline of the proof of the first main theorem}\label{subsec_prel}

In this section, we obtain some preliminary results and give an outline of the proof for our first main theorem.  Let $\cS=(\cP,\cL)$ be a thick generalized quadrangle of order $(s,t)$, and write $\widetilde{E}(x)$ for the full group of symmetries about a point $x$. If the order of $\widetilde{E}(x)$ is even for each point $x$, then Ealy \cite{Ealy} showed that $\cS$ is (isomorphic to) one of $\mW(3,q) $, $\mH(3,q^2)$ and $\mH(4,q^2)$, where $q$ is a power of $2$. Our first main result thus reduces to the next theorem.

\begin{theorem}\label{thm_EalyOdd}
Let $\cS$ be a finite thick generalized quadrangle, and let $r$ be an odd prime. If the full group of symmetries about each point has an order divisible by $r$, then $\cS$ is one of $\mW(3,q) $, $\mH(3,q^2)$ and $\mH(4,q^2)$, where $q$ is a power of $r$.
\end{theorem}

We first fix some notation and record some results from \cite{Ealy} which hold for all primes $r$.  Let $\widetilde{E}(p)$ be the full group  of symmetries about each point $p$, and let $\widetilde{G}=\la \widetilde{E}(x):x\in\cP\ra$. Take a subgroup $E(p)$ of $\widetilde{E}(p)$,  and suppose that there is an odd prime $r$ such that $|E(p)|\equiv 0\pmod{r}$ for each point $p$.  The group $E(p)$ acts freely on points of $\cP\setminus p^\perp$, cf. \cite[Theorem~III.2.3]{Ealy}. For two distinct points $x,y$, by \cite[Lemmas~III.4, III.6]{Ealy} we have
\begin{enumerate}
  \item[(E1)] $E(x)\cap E(y)= \{ \id \}$, and $[E(x),E(y)]= \{ \id\}$ if and only if $x$ and $y$ are collinear.
\end{enumerate}
For a nonempty subset $Y$ of $\cP$, we write $I(Y)=\la E(u): u \in Y\ra$ (and sometimes we will use the same notation in case each $E(u)$ is replaced by the full group of symmetries $\widetilde{E}(u)$). In particular, we set $G=I(\cP)$. For a line $\ell$, we identity $\ell$ with the set of points on it, so that $I(\ell)$ is well defined. We write $G_\ell$ for the stabilizer of $\ell$ in $G$, which clearly fixes the points of $\ell$ setwise. By Lemma III.3.4 and Corollary III.3.6 of \cite{Ealy}, we have the following results:
\begin{enumerate}
  \item[(E2)] $I(p^\perp)$ is transitive on $p^\perp\setminus\{p\}$ and the set of lines through the point $p$, and $I(p^\perp)\unlhd G_p$;
  \item[(E3)] $G_\ell$ is $2$-transitive on the points of the line $\ell$.
\end{enumerate}
It follows  that $G$ is transitive on points, lines, flags and the set of collinear point pairs, cf. \cite[Lemmas~III.3.7-3.9]{Ealy}.  The action of $G$ on $\cP$ is primitive by  \cite[Theorem~III.4.1]{Ealy}. We refer the reader to the Appendix for proofs of the aforementioned results.

\begin{proposition}\label{GisLie}
Let $X=\soc(\widetilde{G})$, the socle of $\widetilde{G}$. The following hold.
\begin{enumerate}
  \item[(1)]The group $G$ is geometrically irreducible on $\cS$,
  \item[(2)]The group $G$ is an almost simple group of Lie type, and we have $\soc(G)=X$.
  \item[(3)] The group $\widetilde{E}(x)$ has odd order for each $x\in\cP$.
\end{enumerate}
\end{proposition}
\begin{proof}
(1) Let $(\cP',\cL')$ be a $G$-invariant substructure. If $\cP'$ contains a point, then it equals $\cP$ by the fact that $G$ is transitive on points. It follows that $\cL'=\cL$. Hence $\cP'$ is empty. Similarly, $\cL'$ is empty by the fact $G$ is transitive on lines. This proves the first claim.

(2) By Theorem \ref{thm_walker4.1}, $G$ is an almost simple group. Write $X=\soc(G)$. It is shown in \cite{spor} that no sporadic almost simple group acts primitively on the points of a generalized quadrangle, so $X$ is not a sporadic group. Assume that $X$ is an alternating group.  Since $G$ is flag-transitive and point-primitive,  we have $X=A_6$ and $\cS$ is the unique generalized quadrangle $\mW(3,2)$ of order $2$ by \cite[Theorem~1.2]{Alt}. The full group of symmetries about a point of $\mW(3,2)$ has order $2$, cf. Section \ref{sec_examples}, and this contradicts our assumption that $E(p)$ has odd order. We conclude that $X$ is a simple group of Lie type. We have $\soc(G)=\soc(\widetilde{G})$ by Theorem \ref{thm_walker4.1}.

(3) Suppose that $\widetilde{E}(p)$ has even order for a point $p$. Since $\widetilde{E}(p^g)=\widetilde{E}(p)^g$ for $g\in X$ and $X$ is transitive on $\cP$, we deduce that $\widetilde{E}(x)$ has  even order for each point $x$. By the main theorem of \cite{Ealy} $\cS$ is one of $\mW(3,q) $, $\mH(3,q^2)$ and $\mH(4,q^2)$, where $q$ is a power of $2$. The group $\widetilde{E}(x)$ at a point $x$ is of order $q$ in each case, cf. Section \ref{sec_examples}. This contradicts the assumption that its order is divisible by an odd prime $r$. This completes the proof.
\end{proof}

Let $X=\soc(\widetilde{G})$ be the socle of $\widetilde{G}$, which is a simple group of Lie type by Proposition \ref{GisLie}. It is transitive on $\cP$ by the fact $\widetilde{G}$ is primitive on points. The above conclusions hold so long as we choose $E(p)$ to be  a subgroup of $\widetilde{E}(p)$ such that $r$ divides $|E(p)|$ at each point $p$. Recall that an {\em $r$-local subgroup} of a finite group is the normalizer of a nontrivial $r$-subgroup. In the next result, we show that we can choose the $E(p)$'s and the odd prime $r$ in a nice way.
\begin{proposition}\label{prop_adjEr}
Let $X=\soc(\widetilde{G})$, and fix a point $p\in \cP$. There is an odd prime $r$ and a subgroup $E(x)$ of $\widetilde{E}(x)$ for each point $x$ such that the following properties hold:
\begin{itemize}
  \item[(a)]$E(x^g)=E(x)^g$ for $x\in \cP$ and $g\in G$, where $G=\la E(x):x\in\cP\ra$;
  \item[(b)]either $G=X$ and $E(p)$ is an elementary abelian $r$-group, or $G=X:E(p)$ is a semidirect product and $E(p)$ is a subgroup of order $r$ in $C_G(X_p)$;
  \item[(c)]$E(p)$ is a minimal normal subgroup of $G_p$;
  \item[(d)]$G_p=N_G(E(p))$, and $G_p$ is an $r$-local maximal subgroup of $G$.
\end{itemize}
\end{proposition}
\begin{proof}
Fix a point $p$ of $\cS$, and let $\cX=\{\widehat{G}:\,X\le\widehat{G}\le\widetilde{G},\widetilde{E}(p)\cap\widehat{G}\ne \{\id\}\}$. Let $\widehat{G}$ be a minimal element of $\cX$ with respect to inclusion, and set $\widehat{E}(x)=\widetilde{E}(p)\cap \widehat{G}$ for $x\in \cP$. We have $\soc(\widehat{G})=X$, and $\widehat{G}$ is primitive on $\cP$ by  \cite[Theorem~III.4.1]{Ealy}.  The action of $X$ on the $\widehat{E}(x)$'s via conjugation is transitive, so the $\widehat{E}(x)$'s are all isomorphic groups. The group $\la \widehat{E}(x):x\in\cP\ra$ has socle $X$ by Proposition \ref{GisLie}, and thus it is in $\cX$. We deduce that $\widehat{G}=\la \widehat{E}(x):x\in\cP\ra$ by the choice of $\widehat{G}$. Since $\widehat{E}(x^g)=\widehat{E}(x)^g$ for $g\in X$ and $x\in\cP$, we have $\widehat{G}=\la X,\widehat{E}(p)\ra$. Both $X_p$  and $\widehat{E}(p)$ are normal subgroups of $\widehat{G}_p$. Also, $\widehat{E}(p)$ has odd order by Proposition \ref{GisLie} (3) and thus is solvable.

We consider two separate cases according as $\widehat{G}=X$ or not. First assume that $\widehat{G}=X$. Let $E(p)$ be a minimal normal subgroup of $X_p$ contained in $\widehat{E}(p)$. Since $\widehat{E}(p)$ is solvable, $E(p)$ is an elementary abelian $r$-group for an odd prime $r$. For $x=p^g\in\cP$ with $g\in X$, let $E(x)=E(p)^g$ (which is well defined by the fact $E(p)\unlhd X_p$). Let $G=\la E(x):x\in\cP\ra$, which is primitive on $\cP$ by  \cite[Theorem~III.4.1]{Ealy}. As in the previous paragraph, we deduce that $G=X$. We have $G_p=N_G(E(p))$ by the fact $G_p$ is maximal in $G$. This proves all the claims in (a)-(d) when $\widehat{G}=X$.

We next assume that $\widehat{G}\ne X$. We deduce that $\widetilde{E}(p)\cap X= \{ \id \}$ by the choice of $\widehat{G}$. It follows that $[\widehat{E}(p),X_p]=\widehat{E}(p)\cap X_p=\{\id\}$. Let $E(p)$ be a subgroup of odd prime order $r$ in $\widehat{E}(p)$. We have $\la X,E(p)\ra= X:E(p)$ and $[E(p),X_p]=\{\id\}$. For $x=p^g\in\cP$ with $g\in X$, let $E(x)=E(p)^g$ (which is well defined). Let $G=\la E(x):x\in\cP\ra$, which is primitive on $\cP$ by \cite[Theorem~III.4.1]{Ealy}. We similarly deduce that $G=X:E(p)$ in this case.  We then have $G_p=X_p E(p)$ by Dedekind's modular law, so $E(p)\le C_G(X_p)$ and $E(p)$ is minimal normal in $G_p$. We have $G_p=N_G(E(p))$ by the fact $G_p$ is a maximal subgroup of $G$, and $E(x^g)=E(x)^g$ for $x\in \cP$ and $g\in G$ by the fact that $E(p)\unlhd G_p$.  This establishes all the claims in (a)-(d) in this case. This completes the proof.
\end{proof}

\begin{proposition}\label{prop_Gell}
Let $r$, $G$ and the $E(x)$'s be as in Proposition \ref{prop_adjEr}, and write $|E(p)|=r^m$ with $m\in\mathbb{N}$. Let $z_0,z_1\ldots,z_s$ be the  points on a line $\ell$ and let $H=\la E(z_i):0\le i\le s\ra$.
\begin{itemize}
  \item[(a)]$H$ is an elementary abelian $r$-group, and $|H|\ge (r^m-1)(s+1)+1$;
  \item[(b)]We have $\{x\in\cP:E(x)\le H\}=\{z_0,\ldots,z_s\}$;
  \item[(c)]We have $G_{\ell}=N_G(H)$.
\end{itemize}
\end{proposition}
\begin{proof}
(a) The group $H$ is an elementary abelian $r$-group, since the $E(z_i)$'s are elementary abelian $r$-groups and they pairwise commute by the property (E1).  Since each $E(z_i)\setminus\{ \id \}$ has size $r^m-1$ and since they are pairwise disjoint by (E1), we have $|H|\ge (s+1)(r^m-1)+1$.

(b) Let $U=\{x\in\cP:E(x)\le H\}$, which clearly contains the $z_i$'s. For $x\in U$, we have  $[E(x),E(z)]= \{ \id \}$ for each $z\in \ell\setminus\{x\}$. By (E1) again,  we deduce that $x$ is collinear with all points of $\ell\setminus\{x\}$. We conclude that $U=\{z_0,\ldots,z_s\}$ as desired.

(c). We have $G_{\ell}\le N_G(H)$, since $G_\ell$ permutes $\{E(z_i):0\le i\le s\}$ by Proposition \ref{prop_adjEr} (a). For $g\in N_{G}(H)$,  it permutes the points on $\ell$ by claim (b) and Proposition \ref{prop_adjEr} (a). It follows that $g$ is in $G_\ell$. Hence $N_G(H)\le G_\ell$, and so equality holds. This completes the proof.
\end{proof}

\begin{corollary}\label{cor_GltGp2}
Let $r$, $G$ and the $E(x)$'s be as in Proposition \ref{prop_adjEr}, and write $|E(p)|=r^m$ with $m\in\mathbb{N}$. For a point $p$, $G_p$ has an elementary abelian $r$-subgroup $H$ such that $|H|\ge(r^m-1)(s+1)+1$ and $[H,E(p)]= \{ \id \}$. Moreover, we have $|G|<|G_p|^2$.
\end{corollary}
\begin{proof}
Let $\ell$ be a line incident with $p$, and let $z_0,\ldots,z_s$ be the points on $\ell$. Let $G_{[\ell]}$ be the pointwise stabilizer of $\ell$ in $G$, and let $H=\la E(z_i):0\le i\le s\ra$. We have $H\le G_{[\ell]}$, and it is the desired subgroup by Proposition \ref{prop_Gell} (a). We have $|G_{[\ell]}|\ge |H|>2(s+1)$. Since $G_\ell$ is $2$-transitive on the points of $\ell$ with kernel $G_{[\ell]}$ by (E3), we deduce that
\[
  |G_\ell|\ge s(s+1)|G_{[\ell]}|>2s(s+1)^2>(s+1)^3.
\]
Since $|G|=(1+s)(1+st)|G_p|$ and $(1+s)|G_p|=(1+t)|G_\ell|$, the claim $|G|<|G_p|^2$  is equivalent to $(1+t) |G_\ell| >(1+s)^2(1+st)$. Since $|G_\ell|>(s+1)^3$, it suffices to show that $(1+t) (s+1) >(1+st)$. The latter inequality clearly holds, and this completes the proof.
\end{proof}

\begin{lemma}\label{lem_stcond}
Let $X=\soc(\widetilde{G})$, and let $r$, $G$ and the $E(x)$'s be as in Proposition \ref{prop_adjEr}. Write $v=|\cP|$, $|E(p)|=r^m$ with $m\in\mathbb{N}$.
\begin{enumerate}
  \item[(a)]We have $t\equiv 0\pmod{r^m}$, $\gcd(s,v)=1$, and $s$ divides both $v-1$ and $|X_p|$.
  \item[(b)]If $r$ is the defining characteristic of $G$, then $(v-1)/s$ is relatively prime to $r$.
\end{enumerate}
\end{lemma}
\begin{proof}
(a) Take a point $p$ and a line $\ell$, and set $H=\la E(z):z\in\ell\ra$. If $x,y$ are distinct points on $\ell$, then $E(x)$ acts freely on the set of lines through $y$ that are distinct from $\ell$. This yields $t\equiv 0\pmod{r^m}$, cf. \cite{tmodrm}. We have $v=(1+s)(1+st)$, so  $v\equiv 1\pmod{s}$ and $\gcd(s,|\cP|)=1$. Since $G_\ell$ is $2$-transitive on the points of $\ell$ by (E3) and $H$ lies in the kernel of this action, we deduce that $s(s+1)|H|$ divides $|G_\ell|$.  It follows from $(1+s)|G_p|=(1+t)|G_\ell|$ that $sr$ divides $|G_p|$.
We have $|G_p|=|X_p|$ or $r|X_p|$ by Proposition \ref{prop_adjEr}, so $s$ divides $|X_p|$. Part (b) follows from the facts $s(1+t+st)=v-1$ and $t\equiv 0\pmod r$. This completes the proof.
\end{proof}

\begin{lemma}\label{lem_GeqXEp}
Let $X=\soc(\widetilde{G})$, and let the odd prime $r$, the $E(x)$'s and the group $G$ be as introduced in Proposition \ref{prop_adjEr}. Fix a point $p$ and a line $\ell$ incident with $p$, and let $H=\la E(z):z\in \ell\ra$. If $G\ne X$,  then we have the following properties.
\begin{enumerate}
  \item[(a)]$G=X:E(p)$, $X_p=C_X(E(p))$ and $G_p=C_G(E(p))=X_p\times E(p)$;
  \item[(b)]$E(p)$ has prime order $r$, and is isomorphic to a subgroup of $\Out(X)$;
  \item[(c)]$X_p\cap H$ is an elementary abelian $r$-subgroup of size at least $\frac{r-1}{r}(s+1)$;
  \item[(d)]If $|X|>|\Out(X)|_{2'}^3$, then $|X|<|X_p|^3$.
\end{enumerate}
\end{lemma}
\begin{proof}
We write $E=E(p)$ in this proof for brevity. By Proposition \ref{GisLie}, we have $\soc(G)=X$. By Proposition \ref{prop_adjEr} and its proof, $|E|=r$, $[E,X_p]=X_p\cap E= \{ \id \}$, $G=X:E$ and $G_p=X_p:E$. The group $E$ does not centralize $X$, since otherwise $E$ would be a minimal normal subgroup of $G$ and thus a subgroup of $\soc(G)=X$. In particular, $E$ acts nontrivially on $X$ via conjugation. We thus have $X_p\le C_X(E)<X$, and $E\cong G/X\le \Out(G)$. It follows that $G_p\le C_X(E)\times E\le C_G(E)<G$. Since $G_p$ is maximal in $G$, we deduce that $G_p=C_X(E)\times E=C_G(E)$. Taking intersection with $X$ on both sides, we obtain $X_p=C_X(E)$.  This proves (a) and (b).

By Proposition \ref{prop_Gell}, $H$ is an elementary abelian $r$-group and $|H|\ge (s+1)(r-1)+1$. Since $H\le G_p$ and $E\le H$, we have $H=H\cap G_p=(H\cap X_p)E$. By the fact $X_p\cap E= \{ \id \}$, we deduce that $|H\cap X_p|=\frac{1}{r}|H|>\frac{r-1}{r}(s+1)$. This proves (c).

Finally, we have $|G|<|G_p|^2$ by Corollary \ref{cor_GltGp2}. By (a), this is equivalent to $|X|<r|X_p|^2$. By (b) we have $r\le |\Out(X)|_{2'}$. It is then straightforward to see that $|X|<|X_p|^3$ if $|X|>|\Out(X)|_{2'}^3$. This proves (d) and completes the proof.
\end{proof}

From now on, we shall take the odd prime $r$ and the $E(x)$'s as introduced in Proposition \ref{prop_adjEr}. Let $G=\la E(x):x\in\cP\ra$, and write $|E(x)|=r^m$ with $m\in\mathbb{N}$. The group $G$ is an almost simple group of Lie type with socle $X=\soc(\widetilde{G})$ by Proposition \ref{GisLie}. The point stabilizer $G_p$ is an $r$-local maximal subgroup of $G$ by Proposition \ref{prop_adjEr}. We write $X={}^t\mathsf{L}(q)'$ in the standard Lie-theoretic notation for Chevalley groups as in \cite[\S~3]{GorLyons}, where $\mathsf{L}\in\{\mathsf{A},\mathsf{B},\ldots,\mathsf{G}\}$ and $q=r_0^f$ with $r_0$ prime. We have $X={}^t\mathsf{L}(q)$ unless $X$ is one of $\mathsf{PSp}_4(2)'$, $\mathsf{G}_2(2)'$, ${}^2\mathsf{F}_4(2)'$ and ${}^2\mathsf{G}_2(3)'$. We write $\texttt{Inndiag}(X)$ for the subgroup of $\Aut(X)$ generated by inner and diagonal automorphisms.\medskip

In the case $G\ne X$, we have summarized the structural results about $G$ and some numerical constraints in Lemma \ref{lem_GeqXEp}. In view of this lemma, we need information about the elements of $\Out(X)$ whose order is an odd prime $r$ for a simple group $X$ of Lie type, cf. \cite[\S 7]{GorLyons}  and \cite[Chapter 4]{GorLySol}. Take an element $g$ of $E(p)$ of order $r$, and regard it as an element of $\Out(X)$, which we can do by Lemma \ref{lem_GeqXEp}. By (7-3) and (7-4) of \cite{GorLyons}, one of the following holds:
\begin{enumerate}
  \item[(R1)] $X=\PSL_{rn}(q)$ or $\PSU_{rn}(q)$ for some $q$ and $n$, with $q\equiv1$ or $-1\pmod{r}$ respectively, and $g$ is an inner-diagonal automorphism;
  \item[(R2)] $r=3$, $X=\mathsf{E}_6(q)$ or ${}^2\mathsf{E}_6(q)$, with $q\equiv1$ or $-1\pmod{3}$ respectively, and $g$ is an inner-diagonal automorphism;
  \item[(R3)] $X$ is defined over $\F_{q}$, where $q=q_0^r$ for some prime power $q_0$, and $g$ is a field automorphism;
  \item[(R4)] $r=3$, $X=\mathsf{P\Upomega}_8^+(q)$ for some $q$, and $g$ is a graph or graph-field automorphism.
  \item[(R5)] $r=3$, $X={}^3\mathsf{D}_4(q)$ for some $q$, and $g$ is a graph automorphism.
\end{enumerate}
We refer to \cite[Chapter 4]{GorLySol} for the conjugacy classes of outer automorphisms of prime order. \medskip

In Theorem \ref{thm_Exceptional} and Theorem \ref{thm_Classical}, we handle the cases where $X=\soc(G)$ is an exceptional group of Lie type and a classical group respectively. Section \ref{secLie} is devoted to the proof of Theorem \ref{thm_Exceptional}, and Section \ref{sec_class} is devoted to the proof of Theorem \ref{thm_Classical}. By combining the results in this section and those two theorems, we complete the proof of Theorem \ref{thm_EalyOdd}.

\section{Proof of the first main result: exceptional groups}\label{secLie}

We take the same notation as in Section \ref{subsec_prel}. In particular, the odd prime $r$ and the $E(x)$'s are as introduced in Proposition \ref{prop_adjEr}, and $X={}^t\mathsf{L}(q)'$ for a prime power $q=r_0^f$ with $r_0$ prime. Let $G=\la E(x):x\in\cP\ra$, and write $|E(x)|=r^m$ with $m\in\mathbb{N}$.
\begin{theorem}\label{thm_Exceptional}
  The socle $X=\soc(G)$ is not a finite exceptional group of Lie type unless it is isomorphic to a finite classical group.
\end{theorem}
The whole of this section is devoted to the proof of Theorem \ref{thm_Exceptional}. We suppose to the contrary that $X={}^d\mathsf{L}(q)'$ is a finite exceptional simple group of Lie type over $\F_q$, where $q=r_0^f$ ($r_0$ prime). We do not need to consider $\mathsf{G}_2(2)'$ or ${}^2\mathsf{G}_2(3)'$ due to the isomorphisms $\mathsf{G}_2(2)'\cong \PSU_3(3)$, ${}^2\mathsf{G}_2(3)'\cong \PSL_2(8)$. For $G={}^2\mathsf{F}_4(2)'$, we let $G_p$ range over its  maximal subgroups as listed in \cite{2F42T, 2F42W}, and check that there is no integer solution to $(1+s)(1+st)=[G:G_p]$ in each case. Hence we assume that $q>2$ when $X={}^2\mathsf{F}_4(q)'$. We write $\mathsf{E}^+_6(q),\mathsf{E}_6^-(q)$ for $\mathsf{E}_6(q),{}^2\mathsf{E}_6(q)$ as per convention, and fix the following notation:
\[
  d=\gcd(2,q-1),\;e_+=\gcd(3,q-1),\,e_-=\gcd(3,q+1),\,f_-=\gcd(4,q+1).
\]

\begin{lemma}\label{lem_ExpGsimple}
The group $G$ is simple, i.e., $G=X$.
\end{lemma}
\begin{proof}
Write $E=E(p)$ and $K=\texttt{Inndiag}(X)$. Suppose to the contrary that $G\ne X$. By Lemma \ref{lem_GeqXEp}, $E$ is cyclic of order $r$,  $X_p=C_X(E)$, $G=X:E$ and $G_p=X_p\times E$. We write $E=\la g\ra$, and regard $g$ as an element of $\Out(X)$.   By (7-3) and (7-4) of \cite{GorLyons}, we have one of the cases (R2), (R3) and (R5) listed at the end of Section \ref{subsec_prel}. The centralizers of $g$'s for (R2), (R5) are available in \cite[Table~4.7.3.A]{GorLySol}. We have $|G_p|>|G|^{1/2}$, i.e., $|X|<r|X_p|^2$, by Corollary \ref{cor_GltGp2}.

We first consider (R2), where $r=3$, $g\in K$, and $X=\mathsf{E}^\epsilon_6(q)$ with $\epsilon=\pm 1$ and  $q\equiv\epsilon\pmod{3}$. It holds that $|X|\ge 3q^{74}$ upon direct check, so we have $|G_p|>3q^{37}$  by the bound $|G_p|>|G|^{1/2}$.  The conditions in \cite[Theorem~1]{LieSax} are satisfied, so $X_p$ is either a parabolic subgroup or appears in  \cite[Table~1]{LieSax}. Comparing with the information about $C_X(g)$ in \cite[Table~4.7.3A]{GorLySol}, we have either $X_p=(\SL_2(q)\circ A_5^\epsilon(q)).d$, or $\epsilon=-1$ and $X_p=(\mathsf{Spin}^-_{10}(q)\circ(q+1)/3).f_-$. We exclude the former case by $|X|<3|X_p|^2$. In the latter case, we have
$|X_p|=\frac{q+1}{3}q^{20}(q^5+1)\prod_{i=1}^4(q^{2i}-1)$ and
$|\cP|=q^{16}\frac{(q^{12}-1)(q^9+1)}{(q^4-1)(q+1)}$. We have $1+s>q^{7.75}$ by Lemma \ref{lem_1psbound}. We deduce from \cite[10-2]{GorLyons} that an elementary abelian $3$-subgroup of $X_p$ has size at most $3^6$, so by Lemma \ref{lem_GeqXEp} we have $3^6\le \frac{2}{3}(1+s)<\frac{2}{3}q^{7.75}$. It holds only if $q=2$, in which case $\gcd(|\cP|-1,|X_p|)=17$. It follows that $s=17$, which contradicts the bound $1+s>q^{7.75}$. This excludes (R2).

We next consider (R5), where  $r=3$, $X={}^3\mathsf{D}_4(q)$ and $g$ is a graph automorphism. It holds that $|X|\ge 3q^{24}$ upon direct check, so we have $|G_p|>3q^{24}$  by the bound $|G_p|>|G|^{1/2}$.  The conditions in \cite[Theorem~1]{LieSax} are satisfied, so $X_p$ is either a parabolic subgroup or appears in  \cite[Table~1]{LieSax}. Comparing with the information about $C_X(g)$ in \cite[Table~4.7.3A]{GorLySol}, we have $X_p=\mathsf{G}_2(q)$. We have $|X_p|=q^{6}(q^2-1)(q^6-1)$ and $|\cP|=q^6(q^8+q^4+1)$. There are polynomials $u(x),v(x)\in\mathbb{Z}[x]$ such that $(|\cP|-1)u(q)+|X_p|v(q)=2^{10}$, so $s\mid 2^{10}$ by Lemma \ref{lem_stcond}. We have $1+s>q^{3.5}$ by Lemma \ref{lem_1psbound}, so $q\le 7$. We have $\gcd(|\cP|-1,|X_p|)\le 2$ for each such $q$, so $s=2$. There is no $q$ such that $1+s>q^{3.5}$: a contradiction. This excludes (R5).

Finally, we consider (R3), where $g$ is a field automorphism of order $r$ and $q=q_0^r$ for some $q_0\in\mathbb{N}$.  It is routine to check that $|X|>|\Out(X)|^3$ for any finite simple exceptional group $X$ of Lie type, so we have $|X|<|X_p|^3$ by Lemma \ref{lem_GeqXEp} (d). By \cite[Theorem~1.6]{AlaMaxExp} and the fact $r$ is odd, we have $r=3$, and $X$ is one of ${}^2B_2(q)$, ${}^2\mathsf{G}_2(q)$ with $q=3^{3k}$,  ${}^2\mathsf{F}_4(q)$ with $q=2^{3k}$, $\mathsf{G}_2(q)$, $\mathsf{F}_4(q)$, $\mathsf{E}_6(q)$, ${}^2\mathsf{E}_6(q)$, $\mathsf{E}_7(q)$ and $\mathsf{E}_8(q)$. They are all excluded by the bound $|X|<3|X_p|^2$, so (R3) does not occur. This completes the proof.
\end{proof}
Suppose that $G$ is simple from now on. By Proposition \ref{prop_adjEr} (d), $G_p=N_G(E(p))$ is a local maximal subgroup of $G$, and such subgroups of finite exceptional groups have been determined in \cite{CLSS}. If $r=r_0$, then $G_p$ is a parabolic subgroup by \cite[3.12]{BT}. We now handle this case. The index of a maximal parabolic subgroup in $G$ can be calculated by using \cite[Proposition~10.7.3]{BCN}, and we do not need further information about the structure of $G_p$. Let  $\Phi_n(x)$ be the cyclotomic polynomial of the complex $n$-th root of unity. By definition  it is the unique irreducible polynomial with integer coefficients that is a divisor of $x^n - 1$ and not a divisor of $x^k - 1$ for any $k\in\mathbb{N}$ such that $1\le k \le n-1$. It is well known that
$\Phi_n(x) = \prod_{\substack{1 \leq k \leq n\\ \mathrm{gcd}(k,n) = 1}}(x - e^{\frac{2\pi ki}{n}})$. We shall write $\Phi_n$ instead of $\Phi_n(q)$ for brevity. Here are the first few $\Phi_i$'s:
\begin{align*}
\Phi_1=q-1,\;\Phi_2=q+1,\;\Phi_3=q^2+q+1,\;\Phi_4=q^2+1,\;\Phi_5=q^4+q^3+q^2+q+1,\\
\Phi_6=q^2-q+1,\;\Phi_8=q^4+1,\;\Phi_9=q^6+q^3+1,\;\Phi_{12}=q^4-q^2+1.
\end{align*}
We refer to \cite[Table~10:2]{GorLyons} for a decomposition of $|G|_{r_0'}$ into a product of $\Phi_i$'s, where $|G|_{r_0'}$ is the largest divisor of $|G|$ relatively prime to $q=r_0^f$. \medskip

Here is a general strategy that we will take.\\

$\left\{
\begin{tabular}{p{0.9\textwidth}}
 We write $|\cP|=f(q)$, $|G_p|=d_0^{-1}h(q)$, where $f(x),h(x)\in\mathbb{Z}[x]$, $f(0)=1$, $h$ is monomial, and $d_0\in\{1,d,e_+,e_-\}$. By the XGCD command in Magma \cite{magma}, we find polynomials $u(x),v(x)\in\mathbb{Z}[x]$, a low degree polynomial $c(x)\in \mathbb{Z}[x]$ such that
\begin{equation}\label{eqn_cq}
   u(q)\left(f(q)-1\right)+v(q)h(q)=c(q),
\end{equation}
and $\deg(c(x))$ equals that of $\gcd(f(x),h(x))\in\mathbb{Q}[x]$. It follows from Lemma \ref{lem_stcond} that $s$ divides $c(q)$. Together with Lemmas \ref{lem_1psbound} and \ref{lem_stcond}, we obtain severe restrictions on $s$. We are usually left to consider a few small values of $q$, so that we can check where $(1+s)(1+st)=|\cP|$ has a solution in $(s,t)$ with $s$ satisfying all those conditions.
\end{tabular}
\right.$

\medskip
\begin{lemma}\label{lem_ExpPara}
If $G$ is simple, then $G_p$ is not a parabolic subgroup of $G$.
\end{lemma}
\begin{proof}
We write $v=|\cP|$ in this proof. Suppose to the contrary that $G_p$ is a parabolic subgroup of $G$. It is well known that its unipotent radical is its Fitting subgroup, so we have $r=r_0$, where $r_0$ is the defining characteristic of $G$. Since $r$ is odd, $G$ is not of type ${}^2\textsf{B}_2$ or ${}^2\textsf{F}_4$. If $G={}^2\mathsf{G}_2(q)$ and $G_p$ is its unique maximal parabolic subgroup up to conjugacy, then $\cP$ has size $q^3+1$ and $G$ acts $2$-transitively on $\cP$, cf. \cite[p.~137]{wilson}. This is impossible, since $G$ cannot map a pair of collinear points to a pair of noncollinear points. Therefore, $G$ is not of type ${}^2\mathsf{G}_2$.

We consider the remaining exceptional groups in what follow. We list the relevant information in Tables  \ref{tab_ExpPara1}-\ref{tab_ExpParaE8}. By Lemma \ref{lem_stcond}, $s$ satisfies the following properties: it has the same $r_0$-part as $v-1$, i.e., $s_{r_0}=(v-1)_{r_0}$; it is relatively prime to $v$, and it divides both $c(q)$ and $|G_p|$, where $c(q)$ is as in \eqref{eqn_cq}. We deduce that $s_{r_0'}$ divides the product of the $\Phi_i$'s (counting multiplicity) that appear in the expression of $|G_p|$ but not in that of $v$. We have a lower bound $b(q)$ on $1+s$ by Lemma \ref{lem_1psbound}, which we list in the respective tables.  Given a candidate $s$ satisfying those conditions, we solve $t$ from $v=(1+s)(1+st)$. It should be in the interval $[\sqrt{s}, s^2]$ and satisfies $s+t\mid st(t+1)$ by \cite{FGQ}. If the pair $(s,t)$ satisfies all those conditions, we say that it is a \ul{feasible pair}.

{\bf We consider $\mathsf{G}_2$, ${}^3\mathsf{D}_4$ and $\mathsf{F}_4$}, where there are at most two sizes for maximal parabolic subgroups. There are five cases to consider, cf. Table \ref{tab_ExpPara1}.
\begin{enumerate}
  \item[(1)]For line 1, we have $s_{r_0}=q$ and $s\mid 5q$, so $s\in\{q,5q\}$.  We deduce from $1+5q>q^{1.25}$ that $q\le 625$. There is no feasible $(s,t)$ pair with this property for each $q$.
  \item[(2)]For line 2, we have $s_{r_0}=q$ and $s\mid 25q$, so $s\in\{q,5q,25q\}$. We deduce from $1+25q>q^{2.25}$ that $q\le 13$, and there is no feasible $(s,t)$ pair for each $q$.
  \item[(3)]For line 3,  we have $v=1+q^3(q^8+q^5+q^4+q+1)$,  so $s_{r_0}=q^3$. Also, $s$ divides both $25q^3$ and $q^{12}(q-1)^2$. It follows that $s\in\{q^3,5q^3,25q^3\}$, and the latter two values occur only if $q\equiv 1\pmod{5}$. If $s=q^3$, then we deduce that $t=q^5+q$, which does not satisfy $s+t\mid st(t+1)$. Assume that $q\equiv 1\pmod{5}$. If $s=5q^3$, then $5^3v\equiv 4q^2 - 20q + 100\pmod{1+s}$ in $\mathbb{Z}[q]$. It follows that $s+1$ divides $4q^2 - 20q + 100$. The latter number is positive, so  $5q^3+1\leq 4q^2 - 20q + 100$, which does not hold for $q\equiv 1\pmod{5}$. If $s=25q^3$, then $5^6v\equiv  24q^2 - 600q + 15000\pmod{1+s}$ in $\mathbb{Z}[q]$, and we exclude it similarly.
  \item[(4)]For line 4, $s_{r_0}=q$ and $s$ divides both $23^4q$ and $q^{24}(q-1)^4$. We thus have  $s=23^iq$ with $0\le i\le 4$, where $i\ge 1$ only if $q\equiv 1\pmod{23}$. We have $1+q<q^{3.75}$ for all $q$, so $q\equiv 1\pmod{23}$. We have $q\le 89$ by the bound $1+23^4q>q^{3.75}$, so $q=47$. We deduce from $1+s>47^{3.75}$ that $s=23^4\cdot 47$, but then $1+s$ does not divide $v$: a contradiction.
  \item[(5)] For line 5, $s_{r_0}=q$ and $s$ divides both $5\cdot 19^4q$ and $q^{24}(q-1)^4$. We have $q^5>1+5q$ for all $q$, so $19\mid s$. It follows that $q\equiv 1\pmod{19}$. We deduce from $1+5\cdot 19^4q>q^5$ that $q\le 27$. There is no such prime power $q$, so this case does not occur.
\end{enumerate}
\begin{table}[h]
\captionsetup{justification=centering}
\caption{Information for $\mathsf{G}_2,{}^3\mathsf{D}_4,{}^2\mathsf{F}_4, \mathsf{F}_4$. In the last column $b(q)$  is a lower bound on $1+s$ given by Lemma \ref{lem_1psbound}.}\label{tab_ExpPara1}
\begin{tabular}{|cccccc|}
\hline
Line & $G$          & $|G_p|$                                    & $v=|\cP|$                                         & $c(q)$         & $b(q)$\\ \hline
1& $\mathsf{G}_2(q)$     & $q^6\Phi_1^2\Phi_2$                        & $\Phi_2\Phi_3\Phi_6$                            & $5q$           & $q^{1.25}$           \\ \hline
2& ${}^3\mathsf{D}_4(q)$ & $q^{12}\Phi_1^2\Phi_2\Phi_3\Phi_6$         & $\Phi_2\Phi_3\Phi_6\Phi_{12}$                   & $25q$          & $q^{2.25}$           \\
3&              & $q^{12}\Phi_1^2\Phi_2\Phi_3$               & $\Phi_2\Phi_3\Phi_6^2\Phi_{12}$                 & $25q^3$        & $q^{2.75}$           \\ \hline
4& $\mathsf{F}_4(q)$     & $q^{24}\Phi_1^4\Phi_2^3\Phi_3\Phi_4\Phi_6$ & $\Phi_2\Phi_3\Phi_4\Phi_6\Phi_8\Phi_{12}$       & $23^4q$        & $q^{3.75}$           \\
5&             & $q^{24}\Phi_1^4\Phi_2^2\Phi_3$             & $\Phi_2^2\Phi_3\Phi_4^2\Phi_6^2\Phi_8\Phi_{12}$ & $5\cdot 19^4q$ & $q^{5}$              \\ \hline
\end{tabular}
\end{table}

{\bf We next consider the cases $G=\mathsf{E}_6(q)$ or ${}^2\mathsf{E}_6(q)$.} Here we handle the maximal parabolic subgroups in the same order as in the corresponding tables in \cite{Craven1}. We list the relevant information in Table  \ref{tab_ExpParaE6}. For line 1 of Table \ref{tab_ExpParaE6}, the lower bound on $1+s$ is too weak for our purpose, so we handle it separately. We have $v=\frac{(q^{12}-1)(q^9-1)}{(q^4-1)(q-1)}$, and $s_{r_0}=(v-1)_{r_0}=q$ in this case. The action of $G$ on $\cP$ is rank $3$, and the two non-trivial suborbits have lengths $a=q(q^3 + 1)\frac{q^8-1}{q-1}$, $b=q^8(q^4 + 1)\frac{q^5-1}{q-1}$ respectively by \cite[4.10.5]{wilson}. They should be the sizes of $p^\perp\setminus\{p\}$, $\cP\setminus p^\perp$ respectively. Since $s_{r_0}=q$ and $t\equiv 0 \pmod{r_0}$, we deduce that $s(t+1)=a$ by considering the respective $r_0$-parts on both sides. On the other hand, we have that $s^2t=b$. It follows that $t_{r_0}=q^6$.  There are $u(x),v(x)\in\mathbb{Z}[x]$ such that $u(q)a+v(q)b=q(q^4+1)$ by Magma \cite{magma}, so $s$ divides $q(q^4+1)$. Taking modulo $q^7$, we obtain from $s(t+1)=a$ that $s\equiv h(q)\pmod{q^7}$, where $h(q)=2q^6 + 2q^5 + 2q^4 + q^3 + q^2 + q$. We have $q^7>h(q)>q(q^4+1)$ for each odd prime power $q$, so there is no feasible $s$. This excludes the case in line 1. We next consider the remaining cases in Table \ref{tab_ExpParaE6}.
\begin{enumerate}
  \item[(1)]For line 2,  $s_{r_0}=q$ and $s$ divides $71^6q$. It follows from $1+s>q^{5.25}$ that $s\ne q$, so $71\mid s$ and $v\equiv 1\pmod{71}$. By this condition and $1+71^6q>q^{5.25}$, we deduce that $q\in\{121,263\}$. For both values we have $(v-1)_{71}=71$, so $s=71q$. We then check that $1+s>q^{5.25}$ does not hold, so this case does not occur.
  \item[(2)]For line 3, $s_{r_0}=q$ and $s$ divides both $43^6q\Phi_5$ and $q(q-1)^6\Phi_5$. It is elementary to deduce that $43\nmid\Phi_5$  for all $q$'s. We deduce from $1+s>q^{6.25}>1+q\Phi_5$ that $43\mid s$, and so $q\equiv 1\pmod{43}$. Let $i$ be the largest integer such that $43^i\mid s$, so that $1\le i\le 5$.   It follows from $1+43^6q\Phi_5>q^{6.25}$ that $q\le 69286460<43^5$. It is routine to check that $v-1\equiv 0\pmod{43^5}$ if and only if $q\equiv 2\cdot 89\cdot 90121\pmod{43^5}$. Since $q<43^5$ and $q$ is odd, we deduce that $i\ne 5$. We then have $1+43^4q\Phi_5>q^{6.25}$, which yields $q\le 168677$. We check that $v-1\equiv 0\pmod{43^4}$ does not hold for such $q$'s, so $i\ne 4$. The bound $1+43^3q\Phi_5>q^{6.25}$ then yields $q\le 8317$. We check that $(v-1)_{43}=43$ for each such $q$, so $i=1$. We similarly obtain $q\le 19$ by $1+43q\Phi_5>q^{6.25}$, which contradicts the fact  $q\equiv 1\pmod{43}$. Hence this case does not occur.
   \item[(3)]For line 4, $s_{r_0}=q$, $s$ divides both $719^6q$ and $q(q-1)^6$, and $1+s>q^{7.25}$. We similarly deduce that $719\mid s$, $q\equiv 1\pmod{719}$, and we have $q\le  547$ by the bound $1+719^6q>q^{7.25}$. There is no such $q$, so this case does not occur.
  \item[(4)]For line 5, $s_{r_0}=q$, $s$ divides both $23^4q$ and $q(q-1)^4\Phi_{10}$, and $1+s>q^{5.25}$. It is elementary to show that $23\nmid \Phi_{10}$ for each $q$. We similarly deduce that $23\mid s$, $q\equiv 1\pmod{23}$, and we have $q\le 19$ by $1+23^4q>q^{5.25}$. There is no such $q$, so this case is impossible.
  \item[(5)]For line 6, $s_{r_0}=q^2$, $s$ divides both $23^4q^2$ and $q^2(q-1)^4\Phi_8$, and $1+s>q^{6}$. It is elementary to show that $23\nmid \Phi_{8}$. We deduce that $23\mid s$, $q\equiv 1\pmod{23}$, and we have $q\le 23$ by $1+23^4q^2>q^{6}$. There is no such $q$, so this case is impossible.
  \item[(6)]For line 7, $s_{r_0}=q$, $s$ divides both $5^4\cdot 19^4q$ and $q(q-1)^4$, and $1+s>q^{7.25}$. We deduce from $1+95^4q>q^{7.25}$ that $q\le 17$, so $s$ divides $5^4q$. The bound $1+5^4q>q^{7.25}$ holds for no odd $q$, so this case does not occur.
  \item[(7)]For line 8, $s_{r_0}=q^2$, $s$ divides both $5^4\cdot 19^4q^2$ and $q^2(q-1)^4$, and $1+s>q^{7.75}$. If $q<11$, then $s=q^2$ and the bound on $1+s$ is violated. If $q\ge 11$, then $1+5^4q^2<1+19^4q^2<q^{7.75}$ holds, so $95$ divides $s$ and $q\equiv 1\pmod{95}$. We deduce from $1+95^4q^2>q^{7.75}$ that $q\le 23$. There is no such $q$, so this case is impossible.
\end{enumerate}

\begin{table}[]
\captionsetup{justification=centering}
\caption{Information for $\mathsf{E}_6,\,{}^2\mathsf{E}_6$. In the last column $b(q)$  is a lower bound on $1+s$ given by Lemma \ref{lem_1psbound}.}\label{tab_ExpParaE6}
\begin{tabular}{|cccccc|}
\hline
Line & $G$ & $e_+\cdot |G_p|$       & $|\cP|$                                                   & $c(q)$                                                    & $b(q)$      \\ \hline
1& $\mathsf{E}_6(q)$ & $ q^{36}\Phi_1^6\Phi_2^4\Phi_3\Phi_4^2\Phi_5\Phi_6\Phi_{8}$       & $\Phi_3^2\Phi_6\Phi_9\Phi_{12}$                           & $13^6q(q^4+1)$                                            & $q^4$      \\
2&          & $ q^{36}\Phi_1^6\Phi_2^3\Phi_3^2\Phi_4\Phi_5\Phi_6$    & $\Phi_2\Phi_3\Phi_4\Phi_6\Phi_8\Phi_9\Phi_{12}$                           & $71^6q$                                                   & $q^{5.25}$ \\
3&         & $q^{36}\Phi_1^6 \Phi_2^3 \Phi_3 \Phi_4\Phi_5$    & $\Phi_2\Phi_3^2\Phi_4\Phi_6^2\Phi_8\Phi_9\Phi_{12}$       & $43^6q\Phi_5$                                  & $q^{6.25}$ \\
4&         & $q^{36}\Phi_1^6 \Phi_2^3 \Phi_3^2$    & $\Phi_2\Phi_3\Phi_4^2\Phi_5\Phi_6^2\Phi_8\Phi_9\Phi_{12}$ & $719^6q$                                                  & $q^{7.25}$ \\ \hline
5&${}^2\mathsf{E}_6(q)$ & $q^{36}\Phi_1^4\Phi_2^5\Phi_3\Phi_4\Phi_6^2\Phi_{10}$ & $\Phi_2\Phi_3\Phi_4\Phi_6\Phi_8\Phi_{12}\Phi_{18}$        & $23^4q$ & $q^{5.25} $ \\
6&             &$q^{36}\Phi_1^4 \Phi_2^4 \Phi_3\Phi_4\Phi_6 \Phi_{8}$& $\Phi_2^2\Phi_3\Phi_4\Phi_6^2\Phi_{10}\Phi_{12}\Phi_{18}$ & $23^4q^2$               & $q^6$      \\
7&             &$q^{36}\Phi_1^4\Phi_2^3 \Phi_3\Phi_4\Phi_6 $&  $\Phi_2^3\Phi_3\Phi_4\Phi_6^2\Phi_8\Phi_{10}\Phi_{12}\Phi_{18}$ & $5^4\cdot 19^4q$   & $q^{7.25}$ \\
8&            &$q^{36}\Phi_1^4 \Phi_2^3 \Phi_3\Phi_4$ & $\Phi_2^3\Phi_3\Phi_4\Phi_6^3\Phi_8\Phi_{10}\Phi_{12}\Phi_{18}$        &          $5^4\cdot 19^4q^2$    & $q^{7.75}$ \\ \hline
\end{tabular}
\end{table}

{\bf We next consider the cases $G=\mathsf{E}_7(q)$.} We list the relevant information in Table \ref{tab_ExpParaE7}, where we follow the same order as in \cite[Table~4.1]{Craven2}. In the last column we give the largest odd $q$ such that $1+c(q)>b(q)$. In all cases, we check by computer that $1+\gcd(v-1,|G_p|)<b(q)$ for the $q$'s that satisfy $1+c(q)>b(q)$. Hence the bound on $1+s$ in Lemma \ref{lem_1psbound} is violated, and we cannot have $G=\mathsf{E}_7(q)$.  \medskip
\begin{table}[]
\captionsetup{justification=centering}
\caption{Information for $\mathsf{E}_7$. The 4th column  is a lower bound $b(q)$ on $1+s$ by Lemma \ref{lem_1psbound}, and the last column is the largest $q$ such that $1+c(q)>b(q)$.}\label{tab_ExpParaE7}
\begin{tabular}{|ccccc|}
\hline
Line &  $d\cdot |G_p|$                                                      & $c(q)$          & $b(q)$   & max $q$    \\ \hline
1&  $ q^{63}\Phi_1^7\Phi_2^6\Phi_3^2\Phi_4^2\Phi_5\Phi_6^2\Phi_{8}\Phi_{10}$                        & $5^{22}\cdot 13^2q$                                            & $q^{8.25}$  &  $263$  \\
2&  $ q^{63}\Phi_1^7\Phi_2^3\Phi_3^2\Phi_4\Phi_5\Phi_6\Phi_{7}$                         & $5^{15}\cdot 23^7q$                                            & $q^{10.5}$ & $127$     \\
3&  $ q^{63}\Phi_1^7\Phi_2^4\Phi_3^2\Phi_4\Phi_5\Phi_6$              & $q^{11.75}$ & $5^8\cdot13^7\cdot 31^7q$ & $163$     \\
4&  $ q^{63}\Phi_1^7\Phi_2^4\Phi_3^2\Phi_4$                             & $10079^7q$                                            & $q^{13.25}$   & $193$   \\
5&  $ q^{63}\Phi_1^7\Phi_2^3\Phi_3^2\Phi_4\Phi_5$          & $29^7\cdot 139^7q$                                            & $q^{12.5}$   & $151$    \\
6&  $ q^{63}\Phi_1^7\Phi_2^5\Phi_3\Phi_4^2\Phi_5\Phi_6\Phi_8$        & $5^8\cdot37^2\cdot 151^7q$                                            & $q^{10.5}$   &$331$   \\
7&  $ q^{63}\Phi_1^7\Phi_2^4\Phi_3^3\Phi_4^2\Phi_5\Phi_6^2\Phi_8\Phi_9\Phi_{12}$          & $5^8\cdot11^7\cdot 13^2q$                                            & $q^{6.75}$  &$421$    \\ \hline
\end{tabular}
\end{table}

{\bf Finally, we consider the case $G=\mathsf{E}_8(q)$.} We list the relevant information in Table \ref{tab_ExpParaE8}, where in the second column we indicate the type of a Lie subgroup in the corresponding parabolic subgroup $G_p$. We exclude all the eight cases using computer by the same approach as in the $\mathsf{E}_7$ case, and this completes the proof.
\begin{table}[]
\captionsetup{justification=centering}
\caption{Information for $\mathsf{E}_8$. The 5th column  is a lower bound $b(q)$ on $1+s$ by Lemma \ref{lem_1psbound}, and the last column is the largest $q$ such that $1+c(q)>b(q)$.}\label{tab_ExpParaE8}
\begin{tabular}{|cccccc|}
\hline
Line & type of $G_p$  &   $|G_p|$            & $c(q)$                & $b(q)$  & max $q$    \\ \hline
1&  $\mathsf{D}_7$  & $q^{120}\Phi_1^8\Phi_2^6 \Phi_3^2\Phi_4^3\Phi_{5}\Phi_6^2 \Phi_{7}\Phi_{8}\Phi_{10}\Phi_{12}$   & $17^8\cdot43\cdot 127^8q$            & $q^{19.5}$  & $31$     \\
2&  $\mathsf{A}_7$  & $q^{120}\Phi_1^8\Phi_2^4 \Phi_3^2\Phi_4^2\Phi_{5}\Phi_6\Phi_{7}\Phi_{8}$   & $37^8\cdot43\cdot 467^8q$           & $q^{23}$ & $41$\\
3& $\mathsf{A}_1\mathsf{A}_6$ &  $q^{120}\Phi_1^8\Phi_2^4 \Phi_3^2\Phi_4^2\Phi_{5}\Phi_6\Phi_{7}$ &$69119^8q$ & $24.5$ & $43$\\
4& $\mathsf{A}_1\mathsf{A}_2\mathsf{A}_4$ &$q^{120}\Phi_1^8\Phi_2^4 \Phi_3^2\Phi_4\Phi_{5}$ &  $483839^8q$ & $26.5$ & $59$\\
5&$\mathsf{A}_4\mathsf{A}_3$ & $q^{120}\Phi_1^8\Phi_2^4 \Phi_3^2\Phi_4^2\Phi_{5}$ & $241919^8q$ & $q^{26}$ & $49$\\
6&$\mathsf{D}_5\mathsf{A}_2$ & $q^{120}\Phi_1^8\Phi_2^5 \Phi_3^2\Phi_4^2\Phi_{5}\Phi_6\Phi_8$ & $197^8\cdot 307^8q$ & $q^{24.25}$& $43$\\
7& $\mathsf{E}_6A_1$ & $q^{120}\Phi_1^8\Phi_2^5 \Phi_3^3\Phi_4^2\Phi_{5}\Phi_6^2\Phi_8\Phi_9\Phi_{12}$ & $109\cdot 6719^8q$ &$q^{20.75}$& $43$\\
8& $\mathsf{E}_7$ & $q^{120}\Phi_1^8\Phi_2^7\Phi_3^3\Phi_4^2\Phi_{5}\Phi_6^3\Phi_7\Phi_8\Phi_9\Phi_{10}\Phi_{12}\Phi_{14}\Phi_{18}$
& $109\cdot 239^8q$ &$q^{14.25}$ & $37$ \\ \hline
\end{tabular}
\end{table}
\end{proof}

By Lemma \ref{lem_ExpPara}, we have $r\ne r_0$, where $r_0$ is the defining characteristic of the simple group $G$ of Lie type. By \cite{CLSS} and the fact $r$ is odd, either $G_p$ is a subgroup of  maximal rank, or $(G,E(p))$ is one of the cases in Table \ref{tab_CLSS1}. Here, $K=\texttt{Inndiag}(X)$ in the table.  The maximal subgroups of maximal ranks in finite exceptional groups of Lie type have been determined in \cite{LSS}. 
\begin{table}[h]
\captionsetup{justification=centering}
\caption{The cases in \cite[Table~1]{CLSS} with $|E|$ odd, where $E=E(p)$ and $K=\texttt{Inndiag}(X)$.}\label{tab_CLSS1}
\begin{tabular}{|ccccc|}
\hline
$L$                 & $E$ & $C_K(E)$  &$N_K(E)/C_K(E)$           & Conditions                         \\ \hline
$\mathsf{F}_4(r_0)$          & $3^3$  & $E$     & $\SL_3(3)$             & $r_0\ge 5$                                             \\
$\mathsf{E}_6^\epsilon(r_0)$ & $3^3$ & $[3^6]$ &  $\SL_3(3)$ &$\epsilon=\pm1,3\mid p-\epsilon,r_0\ge 5$              \\
$\mathsf{E}_8(r_0^a)$        & $5^3$  & $E$ & $\SL_3(5)$                 & $r_0\ne 2,5$, $a=\begin{cases}1,  \textup{ if } 5\mid r_0^2-1,\\ 2,\textup{ if } 5\mid r_0^2+1. \end{cases}$ \\
${}^2\mathsf{E}_6(2)$        & $3^2$  & $E\times \mathsf{G}_2(2)$ &$N_X(E)=\PSU_3(2)\times \mathsf{G}_2(2)$    &                                                        \\ \hline
\end{tabular}
\end{table}

\begin{lemma}\label{lem_ExpCLSS}
If $G=X$ and $r$ is not the defining characteristic $r_0$, then the cases in Table \ref{tab_CLSS1} do not occur.
\end{lemma}
\begin{proof}
Let $E=E(p)$. We have $G_p=N_X(E)$ by Proposition \ref{prop_adjEr}. It is routine to check that $|X|>|N_K(E)|^2$ holds in all cases listed in Table \ref{tab_CLSS1}, and this contradicts the condition $|G|<|G_p|^2$ in Corollary \ref{cor_GltGp2}. This completes the proof.
\end{proof}

\begin{lemma}\label{lem_ExpMaxRank}
If $G=X$ and $r$ is not the defining characteristic $r_0$, then $G_p$ is not a subgroup of  maximal rank in $G$.
\end{lemma}
\begin{proof}
Suppose to the contrary that $G_p$ is a maximal subgroup of  maximal rank in $G$, so that it appears in Tables 1 and 2 of \cite{LSS}.  We have $|G_p|>|G|^{1/2}$ by Corollary \ref{cor_GltGp2}, so $G_p$ appears in \cite[Table~1]{LieSax}. By combining those two results, we see that there are $13$ candidate $(G,G_p)$ pairs. Among them, the nine candidate $(G,G_p)$ pairs
in Table \ref{tab_ExpMaxRank}  are excluded by the fact that their Fitting subgroups are either trivial or $2$-groups.
\begin{table}[h]
\captionsetup{justification=centering}
\caption{The maximal subgroup $G_p$'s of maximal rank excluded by considering Fitting subgroups.}\label{tab_ExpMaxRank}
\begin{tabular}{|c|c|}
\hline
$G$                                 & $G_p$                                                            \\ \hline
$\mathsf{F}_4(q)$                            & ${}^3\mathsf{D}_4(q):3$, \;$\mathsf{Spin}_9(q)$,\; $\mathsf{Spin}_8^+(q).S_3$ \\
$E^\epsilon_6(q)$, $\epsilon=\pm 1$ & $(\SL_2(q)\circ A_5^\epsilon(q)).d$                               \\
${}^2\mathsf{E}_6(q)$                        & $\mathsf{Spin}_{10}^-(q)\circ(q+1)/e_-).f_-$                     \\
$\mathsf{E}_7(q)$                            & $(\SL_2(q)\circ\mathsf{P\Upomega}_{12}^+(q)).d$                     \\
$\mathsf{E}_8(q)$                            & $(\SL_2(q)\circ \mathsf{E}_7(q)).d$,\; $\mathsf{P\Upomega}_{16}^+(q).d$     \\ \hline
\end{tabular}
\end{table}

We consider the remaining four cases:
\begin{enumerate}
\item[(a)]$G=\mathsf{G}_2(q)$ with $q>2$, and $G_p\in \{\SL_3(q).2,\SU_3(q).2\}$,
\item[(b)]$G=\mathsf{E}_7(q)$, and $G\in\{e_\epsilon.(E^\epsilon_6(q)\times (q-\epsilon)/(de_\epsilon)).e_\epsilon.2:\epsilon=\pm 1\}$.
\end{enumerate}
For (a), we deduce that $r=3$ and $3$ divides $q-1$ or $q+1$ respectively by considering the Fitting subgroups. By Corollary \ref{cor_GltGp2}, $G_p$ has an elementary abelian $3$-subgroup $H$ of size at least $2(s+1)+1$. On the other hand, we have $|H|\le 3^2$ by \cite[(10-2)]{GorLyons}. It follows that $|H|=9$, and $s+1\le 4$. We have $|\cP|\in\{q^3(q^3-1),q^3(q^3+1)\}$, and the condition $|\cP|<(1+s)^4$ in Lemma \ref{lem_1psbound} holds for no $q$. This contradiction shows that the cases in (a) do not occur. For (b), we have $r\mid q-\epsilon$, $|E|=r$ and $G_p$ has an elementary abelian $r$-subgroup $H$ such that $|H|\ge (r-1)(s+1)+1$.  Similarly, we have $|H|\le r^7$ by \cite[(10-2)]{GorLyons}. Together with the bound $|\cP|\le (1+s)^4$, we deduce that $|\cP|(q-\epsilon-1)^4<((q-\epsilon)^7-1)^4$. We have $|\cP|=\frac{1}{2}q^{27}(q^5+\epsilon)(q^9+\epsilon)\frac{q^{14}-1}{q-\epsilon}$, and we check that this inequality holds for no $q$ in either case. This excludes (b) and completes the proof.
\end{proof}

$\left\{
\begin{tabular}{p{0.9\textwidth}}
We now summarize what we have done so far in this subsection. Suppose that $G$ is a finite exceptional simple group of Lie type, and let $r_0$ be its defining characteristic. We have shown that $G_p$ is a local maximal subgroup of $G$ in Proposition \ref{prop_adjEr}, and that $G$ is simple in Lemma \ref{lem_ExpGsimple}. We cannot have $r=r_0$, since otherwise $G_p$ would be a maximal parabolic subgroup by \cite{BT}  and we have excluded this possibility in Lemma \ref{lem_ExpPara}.  It follows that $r\ne r_0$, and either $(G,G_p)$ is one of the four pairs listed in Table \ref{tab_CLSS1}, or it is a subgroup of maximal rank by \cite{CLSS}. The former case has been excluded in Lemma \ref{lem_ExpCLSS}, and the latter case has been excluded in Lemma \ref{lem_ExpMaxRank}. This completes the proof of Theorem \ref{thm_Exceptional}.\hspace*{\fill} {\bf QED}
\end{tabular}
\right.$

\medskip
\section{Proof of the first main result: classical groups}\label{sec_class}
\label{secclass}
We continue to use the same notation as in Section \ref{subsec_prel}: the odd prime $r$ and the $E(x)$'s are as introduced in Proposition \ref{prop_adjEr}, $G=\la E(x):x\in\cP\ra$, and write $|E(x)|=r^m$ with $m\in\mathbb{N}$.  The whole of this subsection is devoted to the proof of the following result.
\begin{theorem}\label{thm_Classical}
If the socle $X=\soc(G)$ is a finite classical group of Lie type, then $\cS$ is isomorphic to one of $\mW(3,q) $, $\mH(3,q^2)$ and $\mH(4,q^2)$ for some prime power $q$.
\end{theorem}
Suppose that $X={}^d\mathsf{L}(q)'$ is a finite simple classical group, where $q=r_0^f$ with $r_0$ prime.  If $X=\PSp_4(2)'$ or $\mathsf{P\Upomega}_4^-(3)$, then $X\cong A_6$ and it has been excluded in the proof of Lemma \ref{GisLie} (2) by using \cite[Theorem~1.2]{Alt}, so $X\ne\PSp_4(2)',\mathsf{P\Upomega}_4^-(3)$. Let $V$ be an $n$-dimensional vector space over $\F$ which is the natural (projective) module of $X$. Here, we have $\F=\F_{q^u}$  with $u=2$ if $X$ is unitary and $u=1$ otherwise. \begin{color}{black} Let $\Upomega(V)$ be the quasisimple subgroup of $\GL(V)$ whose quotient group modulo its center is $X$\end{color}. Let $I(V)$ be the isometry group of $V$, and write $\overline{I(V)}$, $\overline{\Upomega(V)}$ for the corresponding quotient groups modulo their respective centers. We have $X=\overline{\Upomega(V)}$ by the fact $X\ne\PSp_4(2)'$. We write $\Upomega=\Upomega(V)$ for brevity if no confusion arises. If $G$ is unitary, symplectic or orthogonal, we write $\upkappa$ for an associated nondegenerate $\Upomega$-invariant Hermitian, alternating or quadratic form on $V$; if $G$ is linear, we let $\upkappa=0$.

Two classical groups can be isomorphic and yet have different natural modules, cf. \cite[p.~46]{KL}, but this will cause no confusion in our arguments below. We shall make use of the following inequalities from \cite{AlaBur}: for natural numbers $a,q\ge 2$ we have
\begin{align}
  (1-q^{-1}-q^{-2})q^{a^2}<&|\GL_a(q)|\le  (1-q^{-1})(1-q^{-2})q^{a^2},\label{eqn_GLineq}\\
  (1+q^{-1})(1-q^{-2})q^{a^2}<&|\GU_a(q)|\le  (1+q^{-1})(1-q^{-2})(1+q^{-3})q^{a^2}.\label{eqn_GUineq}
\end{align}
\begin{lemma}\label{lem_ClaGsimple}
The group $G$ is simple, i.e., $G=X$.
\end{lemma}
\begin{proof}
Write $E=E(p)$ and $K=\texttt{Inndiag}(X)$. We suppose to the contrary that $G\ne X$. By Lemma \ref{lem_GeqXEp}, $E$ is cyclic of order $r$,  $X_p=C_X(E)$, $G=X:E$ and $G_p=X_p\times E$. We write $E=\la g\ra$, and regard $g$ as an element of $\Out(X)$.  Let $\widetilde{g}$ be a preimage of $g$ in $\GL(V)$ of smallest possible order. If $g\in K$ and $r\ne r_0$, then $g^X=g^K$ by \cite[Theorem~4.2.2(j)]{GorLySol}. By (7-3) and (7-4) of \cite{GorLyons}, we have one of the cases (R1), (R3) and (R4) listed at the end of Section \ref{subsec_prel}. We have $|G|<|G_p|^2$, i.e., $|X|<r|X_p|^2$, by Corollary \ref{cor_GltGp2}.

We first consider (R1). We only give details for $X=\PSL_{n}(q)$ here, since the case $X=\PSU_n(q)$ is very analogous. In this case, we have $K=\PGL_n(q)$, $a=\frac{n}{r}\in\mathbb{N}$, $g\in K$, and  $q\equiv1\pmod{r}$.  By \cite[Table~B.3]{BurGiu},  we have three cases:
\begin{enumerate}
  \item[(i)]
 There are integers $a_0\le a_1\le\cdots\le a_{r-1}$ (at least two nonzero) such that $\sum_{i=0}^{r-1}a_i=n$ and $a_j\ne a$ for some $j$, and $C_{K}(g)=\frac{1}{q-1}\prod_j|\GL_{a_j}(q)|$;
  \item[(ii)]
 $\vert C_{K}(g)\vert = \frac{r}{q-1}|\GL_{a}(q)|^r$;
  \item[(iii)]
 $\vert C_{K}(g)\vert = \frac{r}{q-1}|\GL_{a}(q^r)|$.
\end{enumerate}
For (i), $\widetilde{g}$ has order $r$, and it has at least two nontrivial eigenspaces in $V$ which are $G_p$-invariant by \cite[Proposition~3.2.2]{BurGiu}. Since $g$ is an inner-diagonal automorphism, $G_p$ fixes each eigenspace of $\widetilde{g}$. If $V_0,\ldots,V_m$ are the eigenspaces of $\widetilde{g}$, then $G_p$ is a proper subgroup of $G_{V_0}$. This contradicts the fact that $G_p$ is maximal in $G$, so (i) does not occur. For (ii) and (iii), we have $|X_p|=\frac{1}{d}|C_K(g)|$, where $d=\gcd(n,q-1)=\frac{|K|}{|X|}$. We check that $|X|<r|X_p|^2$ does not hold in either case, where we make use of \eqref{eqn_GLineq} when $a\ge 2$. This excludes (R1).

We next consider (R3), where $q=q_0^r$ for some prime power $q_0$, and $g$ is a field automorphism. By \cite[(7-2)]{GorLyons}, there is one conjugacy class of such $E$ in $\Aut(K)$. Therefore, $X_p$ is a classical group of the same type defined over $\F_{q_0}$ up to conjugacy in $G$ and $r\mid f_{2'}$, where $f_{2'}$ is the largest odd divisor of $f$. It is routine to check that $|X|>f_{2'}^3$ for all simple classical groups $X$, so it follows from $|X|<r|X_p|^2$ that $|X|<|X_p|^3$. By \cite[Tables~3.5.H,~3.5.I]{KL} and \cite{BHR}, $X_p$ is a maximal subgroup of $X$. We examine Propositions 4.7 (linear), 4.17 (unitary), 4.22 (symplectic) and 4.23 (orthogonal) of \cite{AlaBur} to deduce that $r=3$, and $X_p=\PSL_n(q_0)$ or $X_p=\PSU_n(q_0)$. Both cases are excluded by the bound $|X|<3|X_p|^2$ by using \eqref{eqn_GLineq} and \eqref{eqn_GUineq}, and we omit the details.

It remains to consider (R4), where $r=3$, $X=\mathsf{P\Upomega}_8^+(q)$, and $g$ is a graph or graph-field automorphism. By \cite[Theorem~9.1]{GorLyons} (see also \cite[Proposition~3.5.23]{BurGiu}), $C_X(g)$ is one of the following: $\mathsf{G}_2(q)$, $\PGL_3(q)$ ($q\equiv 1\pmod{3}$),  $\PGU_3(q)$ ($q\equiv 2\pmod{3}$, $q>2$), $[q^5].\SL_2(q)$ with $q=3^f$, ${}^3\mathsf{D}_4(q_0)$ with $q=q_0^3$. We exclude $\PGL_3(q)$, $\PGU_3(q)$ and ${}^3\mathsf{D}_4(q_0)$ by the bound $|X|<3|X_p|^2$, and exclude $[q^5].\SL_2(q)$ by the fact that $G_p=C_G(g)$ is not maximal in $G$, cf. \cite{Kleide8dim} or \cite[Table~8.50]{BHR}.  It remains to consider the case where $C_X(g)=\mathsf{G}_2(q)$.  By Corollary \ref{cor_GltGp2}, $G_p$ has an elementary abelian $3$-subgroup $H$ of order at least  $2(s+1)+1$. By \cite[(10-2)]{GorLyons}, we have $|H|\le 9$. It follows that $s+1\le 4$, and so $|\cP|<4^4$.  On the other hand, $|\cP|=\frac{1}{d^2}q^6(q^4-1)^2$ with $d=\gcd(2,q-1)$. We have $|\cP|>q^{12}>4^4$ for any prime power $q$: a contradiction. This excludes (R4) and completes the proof.
\end{proof}

From now on we assume that $G$ is a simple classical group of Lie type, i.e., $G=X$.

\begin{lemma}\label{lem_ClasEWdef}
If $G=X$ and $r$ is the defining characteristic $r_0$, then there is a unique $r_0$-subgroup $\widehat{E}$ of $\Upomega(V)$  such that its quotient group in $G$ is $E(p)$ and $|\widehat{E}|=|E(p)|$. Let $W=C_V(\widehat{E})$, where $C_V(\widehat{E})=\{v\in V:h(v)=v\textup{ for each }h\in\widehat{E}\}$.
\begin{enumerate}
  \item[(a)]We have $W\ne 0$, $G_p=G_W$, and $\widehat{E}$ is minimal normal in $\Upomega(V)_W$.
  \item[(b)]If $\upkappa=0$, $\widehat{E}$ acts trivially on $W$ and $V/W$. 
  \item[(c)]If $\upkappa\ne0$, then $W$ is totally singular/isotropic, and $\widehat{E}$ acts trivially on $W$, $W^\perp/W$ and $V/W^\perp$.
\end{enumerate}
\end{lemma}
\begin{proof}
(a) Let $Z$ be the center of $\Upomega(V)$, and write $E=E(p)$ in this proof. The order of $Z$ divides $q+1$ or $q-1$ according as $G$ is unitary or not, and it is relatively prime to $|E|$. The existence and uniqueness of $\widehat{E}$ follow by applying the Schur-Zassenhaus Theorem \cite[Theorem~6.2.1]{FinGrps} to the full preimage of $E$ in $\Upomega(V)$. The claim $W\ne 0$ follows by \cite[Lemma~2.6.3]{FinGrps}. The subspace $W$ is invariant under  $G_p=N_G(E)$. By the maximality of $G_p$, we deduce that $G_p=G_W$, where $G_W$ is the stabilizer of $W$ in $G$.  Since $E$ is minimal normal in $G_p$ by Proposition \ref{prop_adjEr} and its full preimage in $\Upomega(V)_W$ is $Z\times \widehat{E}$, it follows that $\widehat{E}$ is also minimal normal in $\Upomega(V)_W$.

(b) The group $\widehat{E}$ acts trivially on $W$ by definition. Similarly, we have $C_{V/W}(\widehat{E})\ne \{ 0 \}$ by \cite[Lemma~2.6.3]{FinGrps}. Let $U$ be its full preimage in $V$; then $U$ is $G_p$-invariant, and we deduce that $U=V$ because $G_p=G_W$ is a maximal subgroup and $W<U$.

(c) Suppose that $\upkappa\ne 0$. Since $G_p$ is maximal in $G$, we deduce that $W$ is totally singular/isotropic or nondegenerate. The group $G_p$ is not $r_0$-local if $W$ is nondegenerate by \cite[Lemmas~4.1.3-4.1.6]{KL}, so $W$ is totally singular/isotropic. Since $G_p=G_W$, it also stabilizes $W^\perp$. The claim that $\widehat{E}$ acts trivially on $W^\perp/W$ and $V/W^\perp$ follows by similar arguments to those in the proof of (b). 
\end{proof}

\begin{lemma}\label{lem_ClasTotSing}
Suppose that $G=X$ and $r=r_0$, and let $V$ be an $n$-dimensional vector space over $\F$ which is the natural projective module of $X$. 
 Let $\widehat{E}$, $W$ be as introduced in Lemma \ref{lem_ClasEWdef}. Assume that $\upkappa\ne 0$, and let $a=\dim(W)$.
\begin{enumerate}
  \item[(i)] If $a=1$, then $\cS$ is one of $\mW(3,q) $, $\mH(3,q^2)$ and $\mH(4,q^2)$.
  \item[(ii)] If $\upkappa$ is a parabolic quadratic form (in which case $a \leq \frac{n - 1}{2}$) and if $a=\frac{n-1}{2}$, then  $\cS$ is the point-line dual of $\mQ(4,q) $.
  \item[(iii)] If $\upkappa$ is a hyperbolic quadratic form, then $a < \frac{n}{2}-1$.
  \item[(iv)] If $\upkappa$ is an elliptic quadratic form (in which case $a \leq \frac{n}{2} - 1$) and if $a=\frac{n}{2}-1$, then $\cS$ is the point-line dual of $\mQ^-(5,q) $.
\end{enumerate}
\end{lemma}
\begin{proof}
(i) Suppose that $a=1$. The group $G$ is transitive on the totally singular/isotropic subspaces of dimension $1$ \begin{color}{black}by \cite[Lemma 2.10.5]{KL}\end{color}, and we can identify $\cP$ with the set of all $1$-dimensional totally singular/isotropic subspaces. This action has rank $3$, so we deduce that $G$ is distance-transitive on the points of $\cS$. By the classification results of such generalized quadrangles in \cite{DisTrans}, $\cS$ is one of the classical generalized quadrangles. The claim then follows.

(ii) Suppose that $\upkappa$ is a parabolic quadratic form, and assume that $W$ is a maximal totally singular space.  Let $\textsf{G}$ be the collinearity graph of $\cP$ in $\cS$, which is strongly regular by \cite[Theorem 2.2.10]{SRG}.  Let $M$ be the set of maximal totally singular spaces, and let \textsf{$\Upgamma$} be the associated {\em dual polar graph}. \begin{color}{black} The graph $\Upgamma$ has $M$ as vertex set, and two vertices $U_1,U_2$ are adjacent if and only if $U_1\cap U_2$ has codimension $1$ in $U_1$.\end{color} The group $G$ is transitive on $M$ \begin{color}{black}by \cite[Lemma 2.5.10]{BurGiu}\end{color}, so we identify $\cP$ with $M$. For two elements $W',W''$ of $\cP$, they are at distance $i$ in \textsf{$\Upgamma$} if and only if $\dim(W'\cap W'')=a-i$. The group $G$ acts edge-transitively on each distance-$i$ graph of \textsf{$\Upgamma$}, and thus $\textsf{G}$ is one of them by the fact $G$ is transitive on collinear point pairs. If $a=2$, there are only two $G$-orbits on the edges, i.e., $G$ is rank $3$ on the maximal totally singular subspaces. We have $\mathsf{PSp}(4,q) = \mathsf{\Upomega}_5(q)$ by \cite[Theorem 11.3.2]{SRG}, and we deduce that
$\cS$ is the point-line dual of the classical generalized quadrangle $\mQ(4,q)$ if $\mathsf{G} = {\Upgamma}$. If $\textsf{G}$ is the complement of \textsf{$\Upgamma$}, then $(1+s)(1+st)=(1+q)(1+q^2)$, $s(t+1)=q^3$ by \cite[Theorem 2.2.19]{SRG}. Here, the second equality follows by considering the valency of $\textsf{G}$. By Lemma \ref{lem_stcond}, we deduce that $s_{r_0}=q$ and $r_0\mid t$.  It follows that $\gcd(q,t+1)=1$, and thus $q^3\mid s$ by the second equation: a contradiction to $s_{r_0}=q$. Assume that $a>2$, so that \textsf{$\Upgamma$} has diameter at least $3$. The complement of $\textsf{G}$ is a strongly regular graph formed by fusing the edges of the distance-$j$ graphs. By \cite{Muzi} (see also \cite[p.~279]{BCN}), we have $n=7$, and $\textsf{G}$ is the distance $3$-graph of $\Upgamma$ and has valency $q^6$. By \cite[3.2.4]{SRG}, we have $(1+s)(1+st)=(q+1)(q^2+1)(q^3+1)$, and $s(t+1)=q^6$. By Lemma \ref{lem_stcond}, we deduce that $s_{r_0}=q$ and $\gcd(q,t+1)=1$. It follows that $q^6\mid s$ by the second equation: a contradiction.

(iii) Suppose that $\upkappa$ is a hyperbolic quadratic form, and let $n=2m$. We cannot have $a =  m-1$, since $G_p$ is a maximal subgroup of $G=\overline{\Upomega(V)}$. Assume to the contrary that $a=m$. Let $M$ be the $\Upomega(V)$-orbit of $W$, and let \textsf{$\Upgamma$} be the associated {\em half dual polar graph} on $M$. \begin{color}{black} The graph $\Upgamma$ has $M$ as vertex set, and two vertices $U_1,U_2$ are adjacent if and only if $U_1\cap U_2$ has codimension $2$ in $U_1$.\end{color} We identify $\cP$ with $M$. For two elements $W',W''$ of $\cP$, they are at distance-$i$ in \textsf{$\Upgamma$} if and only if $\dim(W'\cap W'')=m-2i$. Let $\textsf{G}$ be the collinearity graph of $\cP$ in $\cS$,  so that $G$ is edge-transitive on it. The collinearity graph of $\cP$ is a distance-$i$ graph of \textsf{$\Upgamma$}, and the latter graph is equivalent to the distance $1$-or-$2$ graph of the dual polar graph of an elliptic quadratic space of dimension $n-1$, cf. \cite[p.~278]{BCN}. We deduce that $n=8$ or $10$ again by \cite{Muzi}, and in both cases \textsf{$\Upgamma$} is strongly regular. It follows that $\textsf{G}$ is either \textsf{$\Upgamma$} or its complement, and $G$ is distance-transitive on $\cP$. This is impossible by the main result of \cite{DisTrans}: a contradiction.

(iv) Suppose that $\upkappa$ is an elliptic quadratic form and $a=m$, where $n=2m+2$. We similarly identify $\cP$ with the set of maximal totally singular subspaces, and suppose that $W',W''$ are collinear if and only if $\dim(W'\cap W'')=m-i$ for some $i$. If $n=4$, then $G$ is $2$-transitive on $\cP$ which is impossible.   Assume that $n\ge 6$. By \cite[Theorem~1.33]{GGG}, we have $|\cP|=\prod_{i=2}^{m+1}(q^i+1)$. Since $|\cP|\equiv 1+q^2\pmod{q^3}$, we have $s_{r_0}=q^2$ by Lemma \ref{lem_stcond}. By \cite[Lemma 1.36]{GGG}, the valency of the collinearity graph of $\cS$ is $s(t+1)=q^{i(i+3)/2}\prod_{k=i+1}^m(q^k-1)\prod_{j=1}^{m-i}(q^j-1)^{-1}$. We have $r_0\mid t$ by Lemma \ref{lem_stcond}, so $s_{r_0}=q^{i(i+3)/2}$. It follows that $i(i+3)/2=2$, i.e., $i=1$. That is, the collinearity graph of $\cP$ is the dual polar graph associated with $(V,\upkappa)$. Since these two  graphs have diameters $2$, $m$ respectively, it follows that $m=2$. We deduce that $\cS$ is the point-line dual of $\mQ^-(5,q) $ as desired. This completes the proof.
\end{proof}

\begin{lemma}\label{lem_ClaParaLinear}
If $G=X$ and $r$ is its defining characteristic $r_0$, then $G$ is not linear, i.e., $\upkappa\ne0$.
\end{lemma}
\begin{proof}
Suppose to the contrary that $G=\PSL_n(q)$, and let $E=E(p)$. We have $\Upomega(V)=\SL_n(q)$ in this case. Let $\widehat{E}$ be as in Lemma \ref{lem_ClasEWdef}, and let $W=C_V(\widehat{E})$, $a=\dim(W)$. Then $G_p$ is the quotient image of $\Upomega(V)_W$ in $G$ by Lemma \ref{lem_ClasEWdef}.  For $A\in\GL_a(q)$, $C\in\GL_{n-a}(q)$, $B\in M_{a,n-a}(q)$, set $m(A,B,C)=\begin{pmatrix} A&B\\0&C\end{pmatrix}$. There is a basis of $V$ such that
\[
  \Upomega(V)_W=\left\{m(A,B,C):A\in\GL_a(q),C\in\GL_{n-a}(q),B\in M_{a,n-a}(q),\det(A)\det(C)=1\right\}.
\]
Its unipotent radical $R=\{m(0,B,0): B\in M_{a,n-a}(q)\}$ is the elementwise stabilizer of $W$ in $\Upomega(V)_W$, and so it contains $\widehat{E}$. The subgroup $\SL_a(q)\circ \SL_{n-a}(q)$ acts on $R$ via conjugation, and this is equivalent to its action on $\F_q^a\otimes\F_q^{n-a}$ via $(u\otimes v)^{(A,B)}=uA^{-\top}\otimes vB$, where $A\in\SL_a(q)$, $B\in\SL_{n-a}(q)$. By \cite[Proposition~2.10.6]{KL} both actions are absolutely irreducible, so by \cite[Lemma~4.4.3 (vi)]{KL} $R$ is a minimal normal subgroup of $\Upomega(V)_W$. We thus have $\widehat{E}=R$.  It is now routine to deduce that the centralizer of $E$ in $G_p$ is $E$ itself. This contradicts the fact that $|C_{G_p}(E)|\ge (1+s)(|E|-1)+1$, cf. Corollary \ref{cor_GltGp2}. This completes the proof.
\end{proof}

\begin{lemma}\label{lem_ClaPara}
If $G=X$ and $r$ is its defining characteristic $r_0$, then $\cS$ is one of $\mW(3,q) $, $\mH(3,q^2)$ and $\mH(4,q^2)$.
\end{lemma}
\begin{proof}
Let $\widehat{E}$ and $W$ be as in Lemma \ref{lem_ClasEWdef}, and set $a=\dim(W)$. Let $u=2$ and $\sigma=q$ in the unitary case, and let $u=1$ and $\sigma=1$ otherwise. We have $|\F|=q^u$. The case where $G$ is linear has been settled in Lemma \ref{lem_ClaParaLinear}, so assume that $\upkappa\ne 0$. Let $f:V\times V\rightarrow\F$ be a nondegenerate $\Upomega(V)$-invariant reflexive sesquilinear form on $V$, where $f$ is the associated bilinear form in the orthogonal case and $f=\upkappa$ otherwise. By Lemma \ref{lem_ClasTotSing}, we may suppose that $a>1$ in all cases (i)-(ii)-(iii)-(iv) and $a\leq \frac{n - 3}{2}$ in the orthogonal cases (ii)-(iii)-(iv), and it is our goal to produce a contradiction in these cases. 

By \cite[Lemma~4.1.12]{KL}, there exist subspaces $Y,V_o$ such that $Y$ has dimension $a$ and is totally singular/isotropic, $W\oplus Y$ and $V_o$ are nondegenerate, and $V=(W\oplus Y)\perp V_o$. We have $W^\perp=W\oplus V_o$. We choose a basis $e_1,\ldots,e_a$ of $W$, a basis $f_1,\ldots,f_a$ of $Y$ such that $f(e_i,f_j)=1$ or $0$ according as $i=j$ or not, and then extend them to a  basis $\cB$ of $V$ by adding a  basis of $V_o$. We have $\Upomega(V)_W=R:D$, where $R$ is its unipotent radical and $D=\Upomega(V)\cap (\GL_a(q^u)\times I(V_o))$ with
\[
\GL_a(q^u)\times I(V_o)=\left\{\diag(P,P^{-\sigma\top},Q):\,P\in\GL_a(q^u), Q\in I(V_o)\right\}.
\]
Take $h=\diag(P,P^{-\sigma\top},Q)\in \GL_a(q^u)\times I(V_o)$ such that $\det(h)=1$. In the symplectic and unitary case, $h$ is in  $D$. In the orthogonal case, $h\in D$  if $\det(P)$ is a square and $Q\in \Upomega(V_o)$, cf. \cite[Lemma~4.1.9]{KL}. The group $R$ acts trivially on $W$, $V/W^\perp$ and $W^\perp/W$, and it contains $\widehat{E}$ by Lemma \ref{lem_ClasEWdef}.  Each element of $R$ is of the form $\eta(A,B,C)$, where
\[
  \eta(A,B,C)=\begin{pmatrix}I_a&A&B\\0&I_a&0\\0&C^\top&I_{n-2a}\end{pmatrix}
\]
for $A\in M_a(\F_{q^u})$ and $B,C\in M_{a,n-2a}(\F_{q^u})$. Let $U_0$ be the elementwise stabilizer of $W^\perp=W\oplus V_o$ in $\Upomega(V)$. Then it stabilizes $V_0^\perp=W\oplus Y$, and it is routine to check that $U_0=\{\eta(A,0,0):\,A^{\top}+\epsilon A^\sigma=0\}$, where $\epsilon=-1$ or $1$ according as $\upkappa$ is alternating or not.  For nonzero vectors $u,v\in W$ and $\alpha\in \F_{q^u}^*$, the linear transformation
\begin{equation}\label{eqn_tauuv}
\tau_{\alpha,u,v}:\,x\mapsto x+\alpha\upkappa(x,u)v-\alpha^\sigma\upkappa(v,x)^\sigma u
\end{equation}
lies in $U_0$. If $\upkappa$ is unitary or alternating, then the nontrivial $\tau_{\alpha,u,u}$'s  are transvections, cf. \cite{Grove}.  If $\la u\ra\ne\la v\ra$ and $\upkappa$ is a quadratic form, then $\tau_{\alpha,u,v}$ is a Siegel transformation. We shall use the following properties in the arguments below:
\begin{itemize}
  \item[(a)]If $g=\eta(A,0,C)\in R$, then $g\in U_0$ and thus $C=0$.
  \item[(b)]$\eta(A,B,C)\eta(A',B',C')=\eta(A+A'+BC'^\top,B+B',C+C')$.
  \item[(c)]$\eta(A,B,C)^h=\eta(P^{-1}AP^{-\sigma\top},P^{-1}BQ,P^{-\sigma}CQ^{-\top})$, for $h=\diag(P,P^{-\sigma\top},Q)\in D$.
\end{itemize}
By the last property and facts on irreducible $\GL_a(q^u)$-modules in characteristic $r_0$, cf. \cite[Section~5.2]{BHR} and \cite{schaffer}, we deduce that $U_0$ is a minimal normal subgroup of $\Upomega_W$.

We claim that $\widehat{E}=U_0$. Recall that $\widehat{E}$ is  minimal normal in $\Upomega(V)_W$ by Lemma \ref{lem_ClasEWdef}. If $n=2a$, then $V_o=0$, $R=U_0$, and the claim follows from the fact $U_0$ is minimal normal in $\Upomega(V)_W$ and $\widehat{E}\le R$. Assume that $n>2a$, and suppose to the contrary that $\widehat{E}\ne U_0$. Then $\widehat{E}\cap U_0= \{ \id \}$, and it follows from (a) and (b) that different $\eta(A,B,C)$'s in $\widehat{E}$ have different $B$'s. Hence $\widehat{E}=\{\eta(A(x),x,C(x)):\,x\in T\}$ for some subspace $T$ of $M_{a,n-2a}(\F)$, where $A(x),C(x)$ are taken from $M_{a,n-2a}(\F)$, $M_{a,a}(\F)$ respectively. By \cite[Proposition~2.10.6]{KL} and our assumption on $a$, $\SL_a(q^u)$ and $\Upomega(V_0)$ are  absolutely irreducible on $W$ and $V_o$ respectively.
As in the proof of Lemma \ref{lem_ClaParaLinear}, we use (c) and tensor module arguments to show that $T=M_{a,n-2a}(q^u)$. By the properties (a) and (b), we deduce that $R=\widehat{E}\times U_0$. It follows that $R$ should be elementary abelian, but this is not the case by the property (b).  This proves the claim $\widehat{E}=U_0$.

We now identify $\cP$ with the $\Upomega(V)$-orbit of $W=C_V(\widehat{E})$. By our assumptions on $a$, $\cP$ is the set of all totally isotropic/singular subspaces of dimension $a$. For $W'\in\cP$, $E(W')$ is the group of symmetries at the point $W'\in\cP$ as in Proposition \ref{prop_adjEr}. We have $E(W')=G[W'^\perp]$ by the previous paragraph, where $G[W'^\perp]$ is the elementwise stabilizer of $W'^\perp$ in $G=\overline{\Upomega(V)}$. By the property (E1), we have $G[W_1^\perp]\cap G[W_2^\perp]=\{ \id \}$ for any distinct elements $W_1,W_2$ of $\cP$. In the symplectic and unitary cases, we deduce that $W_1\cap W_2=\{0\}$ for any distinct $W_1,W_2$  by the existence of transvections, cf. \eqref{eqn_tauuv}. This is impossible under our assumption $a>1$. Therefore, $\upkappa$ must be a quadratic form.  We have $a\ge 2$ and $n\ge 2a+3$ by assumption. Take $\{e_i,e_i':\,1\le i\le a+1\}\subset V$ such that $\la e_1,\ldots,e_{a+1}\ra$ and $\la e_1',\ldots,e_{a+1}'\ra$ are totally singular, $f(e_i,e_j')=1$ or $0$ according as $i=j$ or not, and $W=\la e_1,e_2,\ldots,e_a\ra$. By a similar argument that involves Siegel transformations, we deduce that $\dim(W_1\cap W_2)\le 1$ for distinct $W_1,W_2$ in $\cP$. If $a\ge 3$, then $W'=\la e_1,\ldots,e_{a-1},e_a'\ra$ is in $\cP$ and $W\cap W'=\la e_1,\ldots,e_{a-1}\ra$: a contradiction. Hence we have $a=2$. By the property (E1), distinct elements $W_1,W_2$ of $\cP$ are collinear in $\cS$ if and only if $G[W_1^\perp]$ and $G[W_2^\perp]$ commute. For $W_1=\la u,v\ra\in\cP$, we have $G[W_1^\perp]=\la\tau_{\alpha,u,v}: \alpha\in\F_q\ra$. If $\dim(W_1\cap W_2)=1$, then it is straightforward to check that  $G[W_1^\perp]$ and $G[W_2^\perp]$ commute. It follows that $W_1,W_2$ are collinear in $\cS$ if $\dim(W_1\cap W_2)=1$. Let $W_1=\la e_1,e_2'\ra$, and $W_2=\la e_2,e_3\ra$. Then $W$ is collinear with $W_1,W_2$, but  $(W,W_1)$ and $(W,W_2)$ are in distinct $G$-orbits. This contradicts the fact that $G$ is transitive on collinear point pairs, so this case does not occur. This completes the proof.
\end{proof}

\begin{lemma}
If $G=X$ and $r$ is not the defining characteristic $r_0$, then $G=\PSU_4(2)$ and $\cS=\mW(3,3)$.
\end{lemma}

\begin{proof}
Suppose that $r\ne r_0$, and write $E=E(p)$.  
The maximal subgroup $G_p$ of $G$ is one of the geometric classes $\cC_1$-$\cC_8$ or is of class $\cS$ by \cite{AscMax}, and we follow the refined classification in \cite{KL} here. If $G_p$ is of class $\cC_1$, then its unipotent radical $R$ is an $r_0$-group, and so $[E,R]=E\cap R= \{ \id \}$ by the fact $r\ne r_0$. We deduce that $E\le C_{G_p}(R)\le R$: a contradiction to $r\ne r_0$. Hence $G_p$ is not a parabolic subgroup. Also, $G_p$ cannot be of class $\cS$, since otherwise it would be almost simple and not local. Therefore, $G$ is of geometric class $\cC_2,\ldots,\cC_7$ or $\cC_8$. Since we have $|G_p|>|G|^{1/2}$ by Corollary \ref{cor_GltGp2}, $G_p$ is a large subgroup of the simple group $G$ in the sense of \cite{AlaBur}. We go over Propositions 4.7 (linear), 4.17 (unitary), 4.22 (symplectic) and 4.23 (orthogonal) of \cite{AlaBur} and examine such large subgroups whose Fitting subgroup is not trivial, nor a $2$-group one by one.\medskip

Suppose that $G_p$ is of class $\cC_2$. Then $G_p$ stabilizes a decomposition $\cD:V=V_1\oplus\cdots\oplus V_{d}$, where $\dim(V_i)$ is a constant $k=\frac{n}{d}$.  First suppose that $G$ is linear. We have $k\in\{2,3\}$ by \cite{AlaBur}. If $d=1$, then $|G_p|=(q-1)^{n-1}\cdot n!$, and $|G|<|G_p|^{2}$ holds only if $(n,q)=(2,5)$. The case $G=\PSL_2(5)\cong A_5$ has been excluded in \cite{Alt}, so assume that $d>1$. If $k=2$ and $d\ge 2$, we use the same argument as in the proof of \cite[Lemma~4.10]{AlaBur} to deduce from $|G|<|G_p|^2$ that
\begin{equation*}
  (q-1)(1-q^{-1}-q^{-2})\gcd(d,q-1)^2<4\gcd(2d,q-1)(1-q^{-1})^4(1-q^{-2})^4.
\end{equation*}
It does not hold for any $(q,d)$ pairs with $d\ge 2$, so $k\ne 2$. When $k=3$, we similarly obtain
\begin{equation*}
  (q-1)(1-q^{-1}-q^{-2})\gcd(d,q-1)^2q^{3d^2}<36\gcd(3d,q-1)(1-q^{-1})^6(1-q^{-2})^6,
\end{equation*}
which does not hold for any $(q,d)$ pairs with $d\ge 2$. Hence $G$ is not linear. We next consider the case $\upkappa\ne0$. By \cite{AlaBur}, there are two possible cases:
\begin{itemize}
  \item[(i)]$k=2$ and $V_1,V_2$ are totally singular;
  \item[(ii)]each $V_i$ is nondegenerate.
\end{itemize} 
For (i), we check that $|G|<|G_p|^2$ holds for the candidates in \cite{AlaBur} only if $\upkappa$ is a hyperbolic quadratic form with $d\ge 4$. By \cite[Proposition~4.2.7]{KL}, $r$ divides $q-1$, and $|G_p|\le \gcd(d,2)|\GU_d(q)|$. By Corollary \ref{cor_GltGp2}, $G_p$ has an elementary abelian $r$-subgroup $H$ such that $|H|\ge(r-1)(s+1)+1$. By \cite[(10-2)]{GorLyons}, we have $|H|\le r^d$. Together with the bound $|\cP|<(1+s)^4$ in Lemma \ref{lem_1psbound} and $r\le q-1$, we deduce that $|G|(q-2)^4<((q-1)^d-1)^4|G_p|$. By the bounds on $|G|$ in \cite[Corollary~4.3]{AlaBur} and \eqref{eqn_GLineq}, we deduce that
\[
 q^{d^2-d}(q-2)^4<8\gcd(2,d)(1-q^{-1})(1-q^{-2})((q-1)^d-1)^4.
\]
It holds for no $q$ if $d\ge 6$, so $d=4$. We then have
\begin{align*}
|G_p|&=\frac{2}{\gcd(2,q-1)^2}|\GL_4(q)|=\frac{1}{\gcd(2,q-1)^2}q^6\Phi_1^4\Phi_2^2\Phi_3\Phi_4,\\
|\cP|&=\frac{1}{2}q^6\frac{q^4-1}{q-1}(q^3+1)=\frac{1}{2}q^6\Phi_2^2\Phi_4\Phi_6.
\end{align*}
Since $s$ is relatively prime to $|\cP|$ and divides $|G_p|$ by Lemma \ref{lem_stcond}, we deduce that $s$ is odd and it divides $h(q)=\Phi_1^4\Phi_3$. There are polynomials $u(x),v(x)\in\mathbb{Z}[x]$ such that $u(q)h(q)+v(q)(|\cP|-1)=2\Phi_3$ by the XGCD command in Magma. It follows that $s\mid q^2+q+1$, but this contradicts the bound $1+s<|\cP|^{1/4}$ in Lemma \ref{lem_1psbound}. We conclude that (i) does not occur. For (ii), we use the conditions that $G_p$ is $r$-local ($r$ odd) and $|G_p|>|G|^{1/2}$ to exclude all the  candidate cases in \cite{AlaBur} with the following possible exceptions:
\begin{itemize}
  \item[(c2.1)]$G$ is unitary, $d\ge 2$ and $k=2$;
  \item[(c2.2)]$G$ is unitary, $d=1$ and $k\ge 3$;
  \item[(c2.3)]$G=\mathsf{P\Upomega}_8^+(2)$, $G_p=\mathsf{\Upomega}_2^-(2)^4.2^4.S_4$.
\end{itemize}
For (c2.3), $s$ divides $\gcd(|\cP|-1,|G_p|)=3$ by Lemma \ref{lem_stcond}, and $1+s>10$ by Lemma \ref{lem_1psbound}. This is impossible, so (c2.3) does not occur. For (c2.2), $|G|<|G_p|^2$ holds only when $q=k=3$, or $q=2$ and $k\in\{3,4,5\}$. When $q=k=3$, $s$ divides $\gcd(|\cP|-1,|G_p|)$ and $|\cP|=63$. It follows that $s=2$ and $t=10$, and it does not satisfy $t\le s^2$. We exclude the cases $(q,k)\in\{(2,3),(2,5)\}$ similarly, and obtain $s=t=3$ for $(q,k)=(2,4)$. By \cite[6.2.1]{FGQ}, $\cS$ is isomorphic to $\mW(3,3)$ or its point-line dual $\mQ(4,3)$. Since $\mQ(4,3)$ has no central symmetry, we have $\cS \cong \mW(3,3)$. It remains to consider (c2.1), in which case we have $G_p=[a].\PSU_m(q)^2.[b].S_2$ by \cite[Proposition~4.2.9]{KL}, where $a=\frac{(q+1)\gcd(q+1,m)}{\gcd(q+1,2m)}$, $[a]$ is a group of order $a$ whose group structure is unspecified, and $b=\gcd(q+1,m)$. By \eqref{eqn_GUineq}, we deduce that
\[
(q+1)\gcd(2m,q+1)<4(1+q^{-1})^3(1-q^{-2})^3(1+q^{-3})^4,
\]
which holds only if $q=2$ or $4$. We exclude the cases $(m,q)=(2,2)$ as in (c2.2), so $(m,q)\ne(2,2)$.  The group $[a]$ is the Fitting subgroup of $G_p$, and $a=q+1$ is an odd prime for $q=2,4$. It follows that $r=q+1$ and $|E|=r$. By Corollary \ref{cor_GltGp2}, $G_p$ has an elementary abelian $r$-group $H$ such that $|H|>(s+1)(r-1)+1$.  By \cite[(10.2)]{GorLyons}, we deduce that $|H|\le r^{2(m-1)+1}$, and so $s+1\le\frac{r^{2m-1}-1}{r-1}$. Together with the bound $|\cP|<(s+1)^4$, we obtain $q^4|G|<((q+1)^{2m-1}-1)^4|G_p|$. This condition is satisfied only if $q=2$ and $m\le 5$, or $q=4$ and $m\le 3$. We exclude those cases in the same way as in (c2.3).  This concludes the $\cC_2$ case.\medskip

Suppose that $G_p$ is of class $\cC_3$, so that $G_p$ is a subgroup of $\Gamma\textup{L}_{n/k}(q^{uk})$ for some integer $k>1$. Let $d:=\frac{n}{k}$.  If $d=1$, then we have $G=\PSL_n(q)$ with $n=2$ or $(n,q)\in\{(3,i): i=2,3,5\}$, or $G=\PSU_3(3)$ by \cite{AlaBur}. By using the expression of $|G_p|$ in \cite[Lemma~4.3.6]{KL}, we deduce that $G=\PSL_2(4)$ or $\PSL_3(2)$ by the bound $|G|<|G_p|^2$.  We thus have $(1+s)(1+st)=6$ or $8$, which does not hold for any $s,t\ge 2$. Assume that $d\ge 2$. If $k\ge 3$, we check the candidates in \cite{AlaBur} and conclude that the bound $|G|<|G_p|^2$ holds only if $G=\PSp_6(2)$ and $G_p=\PSp_2(8).3$, and we exclude it by the fact $G_p$ is not local. Therefore, we have $k=2$. By \cite[Section~4.3]{KL} and \cite{AlaBur}, we have three candidates with $k=2$ and $G_p$ local:
\begin{color}{black}
\begin{itemize}
  \item[(c3.1)]$G=\PSL_{2d}(q)$, and $G_p$ is of type $\GL_d(q^2)$,
  \item[(c3.2)]$G=\PSp_{2d}(q)$ with $d\ge 2$, and $G_p$ is of type $\GU_d(q^2)$,
  \item[(c3.3)]$G=\mathsf{P\Upomega}^{\epsilon}_{2d}(q)$ with $d\ge 4$ and $\epsilon=(-1)^d$, and $G_p$ is of type $\GU_d(q)$.
\end{itemize}
\end{color}
We exclude (c3.2) by the bound $|G|<|G_p|^2$. For (c3.1),  we have $G_p=[a].\PSL_d(q^2).b.2$ by \cite[Proposition~4.3.6]{KL}, where $a=\frac{(q+1)\gcd(q-1,d)}{\gcd(q-1,2d)}$, $b=\frac{\gcd(q^2-1,d)}{\gcd(q-1,d)}$. Its Fitting subgroup is the cyclic subgroup $[a]$ and an odd prime divisor of $a$ divides $q+1$. Since $E$ is mininal normal in $G_p$ and is a subgroup of $[a]$, we deduce that $r$ divides $q+1$ and $|E|=r$.  By the bound \eqref{eqn_GLineq}, we deduce from $|G|<|G_p|^2$ that
\[
  (q-1)\gcd(q-1,2d)(1-q^{-1}-q^{-2})<4(1-q^{-2})^2(1-q^{-4})^2,
\]
which holds only if $q\le 4$. Since $r$ is odd and divides $q+1$, we have $q\in\{2,4\}$ and $r=q+1$. By Corollary \ref{cor_GltGp2}, $G_p$ has an elementary abelian $r$-group $H$ such that $|H|>(s+1)(r-1)+1$.  By \cite[(10.2)]{GorLyons}, we deduce that $|H|\le r^{m}$, and so $s+1\le\frac{r^{m}-1}{r-1}$. Together with the bound $|\cP|<(s+1)^4$, we obtain $q^4|G|<((q+1)^{m}-1)^4|G_p|$. This holds only if $G=\PSL_4(2)$, in which case we deduce that $(s,t)=(17,79)$. It does not satisfy $r\mid t$, cf. Lemma \ref{lem_stcond}, so this case is impossible. For (c3.3), we have $G_p=(q+1)/a.\PSU_d(q).[b\gcd(q+1,d)]$, where $a=\gcd(q+1,3-(-1)^d)$, $b=\gcd(2,d)$ if $q$ is even and $b=1$ otherwise. We similarly deduce that $r\mid(q+1)_{2'}$ and $|E|=r$. By Corollary \ref{cor_GltGp2}, $G_p$ has an elementary abelian $r$-subgroup $H$ such that $|H|\ge(r-1)(s+1)+1$. By \cite[(10-2)]{GorLyons}, we deduce that $|H|\le r^d$. Together with the bound $|\cP|^{1/4}<s+1$ in Lemma \ref{lem_1psbound}, we obtain $|G|(r-1)^4<(r^d-1)^4|G_p|$. By the estimates of $|G|$ and $\PSU_d(q)$ in \cite[Corollary~4.3]{AlaBur} and the fact $r\le q+1$, we obtain
\[
  q^{d^2-d+5}\gcd(q+1,3-(-1)^d)<8(q+1)\gcd(q+1,d)\gcd(2,d)((q+1)^d-1)^4.
\]
It holds only if $d=4$, or $d=5$ and $q\le 4$, or $d\in\{6,7\}$ and $q=2$. For $d\ge 5$, we use the same arguments as in handling (c2.2) to deduce that $(d,q)=(5,2)$. It follows that $r=3$, $G=\mathsf{P\Upomega}_{10}^-(2)$. There is a unique solution $(s,t)=(33,543)$, but $s(t+1)$ does not divide $|G_p|$: a contradiction to the fact that $G_p$ is transitive on the points collinear with $p$. For $d=4$, we have
\begin{align*}
|G_p|&=\frac{2}{\gcd(2,q-1)^2}q^6\prod_{i=1}^4(q^i-(-1)^i)=\frac{2}{\gcd(2,q-1)^2}q^6\Phi_1^2\Phi_2^4\Phi_6\Phi_4,\\
|\cP|&=\frac{1}{2}q^6(q^3-1)\frac{q^4-1}{q+1}=\frac{1}{2}q^6\Phi_1^2\Phi_3\Phi_4.
\end{align*}
Since $s\mid\gcd(|\cP|-1,|G_p|)$ and $|\cP|$ is a multiple of $2q$, $s$ divides $h(q):=\Phi_2^4\Phi_6$. By the XGCD command in Magma, there are polynomials $u(x),v(x)\in\mathbb{Z}[x]$ such that $u(q)h(q)+v(q)(|\cP|-1)=2\Phi_6$. Therefore, $s$ divides $\Phi_6=q^2-q+1$. This leads to $|\cP|<\Phi_6^4$ by Lemma \ref{lem_1psbound}. It holds for no $q$, so this case does not occur. This excludes the $\cC_3$ case.\medskip

The other geometric classes are more easily dealt with.
\begin{itemize}
\item If $G_p$ is of class $\cC_4$, then $G=\mathsf{\Upomega}_{n}^+(2)$ with $n\in\{8,12\}$, $G_p$ is of type $\Sp_2(2)\times\Sp_{n/2}(2)$ by \cite{AlaBur}, and we exclude them by the bound $|G|<|G_p|^2$.
\item If $G_p$ is of class $\cC_5$, by \cite[Section~4.5]{KL} and \cite{AlaBur} we have one of the following $r$-local candidates: (1) $G=\PSL_2(2^k)$ with $k\in\{2,3\}$, $G_p=\PGL_2(2)$, $r=3$; (2) $G=\PSU_3(2^3)$, $G_p=\PGU_3(2)$, $r=3$. We exclude all of them by the  bound $|G|<|G_p|^2$.
\item If $G_p$ is of  class $\cC_6$, $(G,G_p)$  is one of $(\PSL_3(4),3^2.Q_8)$, $(\PSU_3(5),3^2.Q_8)$  by \cite{AlaBur}. In both cases, $s\mid 72$ and $9\mid t$ by Lemma \ref{lem_stcond}. It is routine to check that $(1+s)(1+st)=[G:G_p]$ does not hold for such $(s,t)$ pairs in either case.
\item If $G_p$ is of class $\cC_7$, then $G=\mathsf{\Upomega}_8^+(2)$ and $G_p=\Sp_2(2)\wr S_3$  by \cite{AlaBur}, and we exclude it by the bound $|G|<|G_p|^2$.
\item If $G_p$ is of class $\cC_8$ and appears in \cite[Table~4.8.A]{KL}, then its Fitting subgroup is not trivial or a $2$-group only if $(G,G_p)=(\PSL_3(4),\PSU_3(2))$, and we exclude it by the bound $|G|<|G_p|^2$.
\end{itemize}
This completes the proof of Theorem \ref{thm_Classical}. By the outline in the last paragraph of Section \ref{subsec_prel}, we have now completed the proof of Theorem \ref{thm_EalyOdd}.
\end{proof}

\medskip
\section{The general Ealy problem}\label{secgen}

In this section, we consider the following general Ealy problem: characterize such finite thick finite generalized quadrangles that there is a central symmetry about each point. Our main theorem in this section is the following.
\begin{theorem}[General Ealy problem]\label{Ealygen}
Let $\cS$ be a finite thick generalized quadrangle. If the full group $\wE(u)$ of symmetries about each point $u$ is nontrivial, then $\cS$ is one of $\mW(3,q) $, $\mH(3,q^2)$ and $\mH(4,q^2)$ for a prime power $q$.
\end{theorem}
This whole section is devoted to the proof of Theorem \ref{Ealygen}. Let $\mS = (\mP,\mL)$ be a minimal counterexample to Theorem \ref{Ealygen} with respect to inclusion. That is, the full group $\vert \wE(x) \vert $ of symmetries about each point $x$ is nontrivial, but $\cS$ is not one of $\mW(3,q) $, $\mH(3,q^2)$ and $\mH(4,q^2)$ for a prime power $q$. Suppose that $\mS$ has parameters $(s,t)$.
Let $G=\la \wE(x):x\in\cP\ra$. For a point $u$,  let $G_{[u]}$ be the linewise stabilizer of $u$ in $G$, i.e.,  the subgroup of $G$ that fixes each line through $u$. Let  $x,y$ be two noncollinear points of $\cS$. A {\em homology} with centers $\{x,y\}$ is an automorphism of $\cS$ that fixes all the lines that contain $x$ or $y$. All the homologies with centers $\{x,y\}$ in $G$ form  a subgroup, which we denote by $G_{[x, y]}$. By definition, we have $G_{[x, y]}=G_{[x]}\cap G_{[y]}$. For a prime $r$, we define
\begin{equation}\label{eqn_Prdef}
  \cP_r=\{y\in\cP:\vert \wE(y) \vert \equiv 0\mod{r}\}.
\end{equation}
Let $\cL_r$ be the set of lines that contain at least two points of $\cP_r$. If $\cP_r$ is not empty, then we call its points {\em $r$-points}. We define
\begin{equation}\label{eqn_Rdef}
  \mathscr{R} = \{r\textup{ prime}:\,\cP_r\ne\emptyset\}.
\end{equation}
We have $\cP=\cup_{r\in \mathscr{R}}\cP_r$, and $\cP_r\ne\cP$ for each $r\in \mathscr{R}$ by our first main theorem and the assumption that $\mS$ is a counterexample to Theorem \ref{Ealygen}.

\begin{lemma}\label{lem_cover}
Let $u,v$ be two noncollinear points, and let $g$ be a nontrivial symmetry with center $u$. Then $v^g$ is in $\{u,v\}^{\perp\!\perp}$.
\end{lemma}
\begin{proof}
Let $\ell_0,\ell_1,\ldots,\ell_t$ be the lines incident with $u$. For $0\le i\le t$, let $z_i$ be the point on $\ell_i$ that is incident with $v$ and let $\ell_i'$ be the line incident with both $v$ and $z_i$. We have $\{u,v\}^{\perp}=\{z_0,z_1,\ldots,z_t\}$. The element $g$ maps $\ell_i'$ to a different line  $\ell_i''$ through $z_i$, and thus $v^g$ is collinear with $z_i$  for each $i$. This completes the proof.
\end{proof}

\begin{lemma}\label{lem_Omegauv}
For two noncollinear points $u,v$ in $\mP_r$, the $\langle \wE(u), \wE(v) \rangle$-orbit that contains $u$ also contains $v$.
\end{lemma}
\begin{proof}
We write $I(u,v)=\langle \wE(u), \wE(v) \rangle$ in this proof. Let $\Upomega$ be the  $I(u,v)$-orbit that contains $u$. The elements of $\Upomega$ are collinear with those in $\{u,v\}^\perp$, so $\Upomega$ is a subset of $\{u,v\}^{\perp\!\perp}$. In particular, no two points of $\Upomega$ are collinear. Since $\wE(u)$ acts freely on $\mP\setminus u^\perp$, we deduce that $|\Upomega|\equiv 1\pmod{r}$ by considering the $\wE(u)$-orbits on $\Upomega$.  Similarly, we have $|\Upomega|\equiv [\![v\in\Upomega]\!]\pmod{r}$ (where $[\![ \cdot ]\!]$ is the Iverson bracket) by considering the $\wE(v)$-orbits on $\Upomega$. It follows that $v$ is in  $\Upomega$.  This completes the proof.
\end{proof}

\begin{lemma}\label{lemtransr}
Let $u$ be any point in $\mS$, and let $r$ be a prime such that $(\cP_r\cap u^\perp)\setminus \{ u\}\ne\emptyset$. Then $G_u$ acts transitively on the $r$-points in $u^\perp \setminus \{ u\}$.
\end{lemma}
\begin{proof}
Let $v,w$ be two distinct $r$-points in $u^\perp \setminus \{ u\}$. If $v,w$ are not collinear $r$-points, then $\langle \wE(v), \wE(w) \rangle$ has an element $g$ which sends $v$ to $w$ by Lemma \ref{lem_Omegauv}. It is clear that $g$ stabilizes $u$, i.e., $g\in G_u$. Suppose that $v,w$ are on the same line $\ell$ through $u$. Take a point $x$ of $u^\perp$ which is not on $\ell$. For any nontrivial element $h\in\widetilde{E}(x)$, the line $\ell^h$ is distinct from $\ell$ and contains the $r$-point $v_1=v^h$. We can find $g_1,g_2\in G_u$ such that $v^{g_1}=v_1$ and $v_1^{g_2}=w$ as in the previous case, so that $g_1g_2$ is the desired element of $G_u$ that maps $u$ to $w$. This completes the proof.
\end{proof}

In the next lemma, We prove some properties of $\mP_r$ and $\mL_r$ that we shall use.

\begin{lemma}\label{lem_PrProp}
Let $r$ be a prime such that $\cP_r\ne\emptyset$, where $\cP_r$ is as in \eqref{eqn_Prdef}.
\begin{enumerate}
  \item[(1)]$\cP_r$ has size at least $3$ and is not contained in a line.
  \item[(2)]$G$ is transitive on $\cP_r$; if furthermore $\cL_r\ne\emptyset$, then $G$ is transitive on $\cL_r$, and $G_\ell$ is $2$-transitive on $\cP_r\cap\ell$ for each line $\ell\in\cL_r$.
  \item[(3)]There is a point $v\in\cP\setminus\cP_r$ which is not collinear with all points of $\cP_r$.
  \item[(4)]If $\cL_r\ne\emptyset$, then there are at least two lines in $\cL_r$ incident with any point of $\cP_r$.
\end{enumerate}
\end{lemma}
\begin{proof}
(1) Take a point $u\in \cP_r$. For a point $x$ not collinear with $u$,  $u^{\wE(x)}$ has size at least two and consists of noncollinear points. Since its elements are $r$-points, it follows that $\mP_r$ is not contained in a line. If $\cP_r=\{u,v\}$ and has size $2$, then take a point $x$ collinear with $v$ but not $u$, and we derive a contradiction by considering the size of $u^{\wE(x)}$. Hence we have $|\cP_r|\ge 3$.

(2) Take two points $v,w$ of $\cP_r$. If $v,w$ are not collinear, there is $g\in \la\wE(v),\wE(w)\ra$ such that $v^g=w$ by Lemma \ref{lem_Omegauv}.  If $v,w$ are on the same line $\ell$, then there is $g\in G_u$ such that $v^g=w$ for a point $u\in\ell\setminus\{v,w\}$ by Lemma \ref{lemtransr}. This proves the first claim.  Suppose that $\cL_r\ne\emptyset$, and take a point $u\in \cP_r$.  Then $G_u$ is transitive on the set of lines of $\cL_r$ incident with $u$ by Lemma \ref{lemtransr}, and so $G$ is transitive on $\cL_r$ by the first claim. The last claim follows from Lemma \ref{lemtransr}.

(3) Let $Y=\cP\setminus\cP_r$, and set $c=|\cP_r|$, $d=|Y|$. Suppose to the contrary that all points of $\cP_r$ are collinear with all points of $Y$. We have $c,d\ge 3$ by (1), so $c\le |Y^\perp|\le t+1$ and similarly $d\le t+1$. It follows that $(1+s)(1+st) = \vert \mP_r \vert + \vert Y \vert \le 2(t+1)$ which is a contradiction to $s>1$.

(4) Suppose that $\cL_r\ne\emptyset$, and take $u\in\cP_r$. There is at least one line $\ell\in\cL_r$ through $u$ by the fact that $G$ is transitive on $\mP_r$. For a point $v\in u^\perp\setminus\ell$, a nontrivial element of $\wE(v)$ maps $\ell$ to a distinct line in $\cL_r$ that is incident with $u$.  This completes the proof.
\end{proof}

We next establish the existence of certain automorphisms of $\cS$. We recall that $G_{[u]}$ is the linewise stabilizer of the point $u$, and a homology with centers $\{x,y\}$ is an element of $G_{[x]}\cap G_{[y]}$ for two noncollinear points $x,y$.
\begin{lemma}\label{lem_GbruExist}
For each point $u$ of $\cS$, $\wE(u)$ is a proper subgroup of $G_{[u]}$.
\end{lemma}
\begin{proof}
Suppose that $u$ is an $r$-point for the prime $r$. By Lemma \ref{lem_PrProp}, there is a point $v\in \cP_r$ which  is not collinear with $u$, and $m=|\wE(x)|$ is a constant for $x\in \cP_r$. Let $X=\{u,v\}^{\perp\!\perp}$ (and note that this set consists of noncollinear points). Let $X_r=\cP_r\cap X$ and $A=\la\wE(x):x\in X\ra$. By Lemma \ref{lem_Omegauv}, we deduce that $A$ is transitive on $X_r$. There are at least $(|X_r|-1)\cdot(m-1)+(|X|-|X_r|)\cdot 1$ nontrivial point-symmetries in $A\setminus\wE(u)$ (where the last term is the trivial lower bound for the number of nontrivial symmetries with center in $X \setminus X_r$), and each maps $u$ to a point of $X_r\setminus\{u\}$.
If $m>2$, then
there are at least two point-symmetries $\beta,\beta'$ with distinct centers in $X\setminus\{u\}$ such that $u^\beta=u^{\beta'}$. By the point-line dual of \cite[(8.1.3)]{FGQ}, $\beta'\beta^{-1}$ is not a point-symmetry. It stabilizes $u$ and each point in $\{u,v\}^\perp$, so it lies in $G_{[u]}$. This establishes the claim when $m>2$.

It remains to consider the case $m=2$, i.e., $r=2$ and $|\wE(x)|=2$ for $x\in \cP_2$. Let $X$, $A$ and $X_2$ be as in the previous paragraph. Then $\vert \widetilde{E}(y) \vert > 2$ for $y \in X \setminus X_2$, and the claim follows by the same arguments. This completes the proof.

\end{proof}

\begin{lemma}\label{lem_homology}
Let $x,y$ be two noncollinear points, where $x$ is an $r$-point and $y$ is an $r'$-point with $r\ne r'$. Then there is a nontrivial homology with both centers in $\{x,y\}^{\perp\!\perp}$.
\end{lemma}
\begin{proof}
Let $X_r$ be the sets of $r$-points in $\{x,y\}^{\perp\!\perp}$, and define $X_{r'}$ similarly. We write $k_r=|X_r|$ and $k_{r'}=|X_{r'}|$. Let $H=\la \widetilde{E}(u): u \in \{x,y\}^{\perp\!\perp}\ra$, which stabilizes $\{x,y\}^{\perp\!\perp}$. By Lemma \ref{lem_Omegauv}, we deduce that $H$ is transitive on both $X_r$ and $X_{r'}$. We have $k_r\equiv 1\pmod{r}$, $k_{r}\equiv 0\pmod{r'}$ by considering the $\wE(x)$-orbits and the $\wE(y)$-orbits on $X_r$ respectively. It follows that $k_r\ge 3$. We obtain similar congruences for $k_{r'}$, from which we deduce that $k_{r'}\ge 3$ and $k_r\ne k_{r'}$. If $H$ is not faithful on $X_r$ or $X_{r'}$, then a nontrivial element in the corresponding kernel is a homology. Suppose that $H$ is faithful on both sets. If a nontrivial element $g$ of $H$ stabilizes two points of $X_r$ or $X_{r'}$, then $g$ is a  homology.  It thus suffices to consider the case where $H$ is a Frobenius group on both $X_r$ and $X_{r'}$.  Since a finite group can only give rise to one Frobenius group up to equivalence, $(H, X_r)$  and $(H, X_{r'})$ are equivalent (the Frobenius kernel is the Fitting subgroup of $H$ \cite[Corollary 17.5]{Pass}, and all Frobenius complements are conjugate). It follows that $\vert X_r \vert = \vert X_{r'} \vert$, a contradiction. This completes the proof.
\end{proof}

\begin{color}{black}
\begin{lemma}\label{lem_Genkey}
Let  $\ell$ be a line through a point $v$, and assume that it contains an $r$-point $u$ distinct from $v$. Then $\ell \setminus\{v\}$ contains a second $r$-point.
\end{lemma}
\begin{proof}


Let $\upalpha$ be a nontrivial element in $G_{[v]} \setminus \widetilde{E}(v)$ (such an element exists by Lemma \ref{lem_GbruExist}), and let $b$ be a point in $v^{\perp} \setminus \{ v\}$ which is not collinear with $u$, and such that $b^\upalpha \ne b$. If no such point exists, then consider any point $x\in \ell \setminus\{v\}$, and choose a point $y\in v^\perp\setminus \ell$. The set $\{x,y\}^{\perp\!\perp}$ has size at least three by considering the $\wE(x)$-orbits on $\{x,y\}^{\perp\!\perp}$. All its points except $x$ are in $v^\perp\setminus \ell$ and are thus fixed by $\upalpha$. We deduce that $\upalpha$ stabilizes $\{x,y\}^\perp$, and so it also stabilizes $\{x,y\}^{\perp\!\perp}$. It follows that $\upalpha$ fixes $x$. Since $x$ was arbitrary, $\upalpha$ fixes all points of $v^\perp$: a contradiction to the fact that $\upalpha\not\in\wE(v)$. So we can indeed find a point $b \in v^{\perp} \setminus \{ v\}$ for which $b^{\upalpha} \ne b$. 
Let $\jmath$ be any line incident with $u$ and different from $\ell$; then the projection $c$ of $b$ on $\jmath$ is different than the projection $c'$ of $b^{\upalpha}$ on $\jmath$, and $\{ v, c \}^{\perp} \cap \{ v, c' \}^{\perp} = \{ u \}$. Now consider the orbits $\Upomega := u^{\widetilde{E}(b)}$
and $\Upomega' := u^{\widetilde{E}(b^{\alpha})}$; then $\Upomega \subseteq \{ v, c \}^{\perp}$ and $\Upomega' \subseteq \{ v, c' \}^{\perp}$, and $\Upomega \cap \Upomega' = \{ u \}$. Also, since $\upalpha$ fixes all lines incident with $v$, $\widetilde{E}(b)$ and $\widetilde{E}(b^{\alpha})$ have the same action on $\ell$ (that is, produce the same line orbits when they act on $\ell$). It follows that for every point $\widehat{u} \in \Upomega \setminus \{ u \}$,
$\widehat{u}' := v\widehat{u} \cap \Upomega'$ is an $r$-point on $v\widehat{u}$ which is different from the $r$-point $\widehat{u}$.
By Lemma \ref{lem_PrProp} (2), we deduce that $\ell \setminus\{v\}$ contains at least two $r$-points: a contradiction.
\end{proof}
\end{color}

\begin{remark}{\rm
Lemma \ref{lem_Genkey} plays a key role in understanding the structures of $\cP_r$ and $\cL_r$. The cohomological nature of this lemma allows us to ``unfold" a hypothetical complete set of $r$-points in $v^{\perp} \setminus \{ v\}$ which hypothetically meets each line on $v$ in at most one point. We observe that the existence of $g\in G_{[v]}\setminus\wE(v)$ is essential to control $\Upomega_1 \cap \Upomega_2$ in the proof. }
\end{remark}

The following result is a direct consequence of Lemmas \ref{lem_PrProp} and \ref{lem_Genkey}.
\begin{corollary}\label{cor_LrProp}
Let $r$ be a prime such that $\cP_r\ne\emptyset$, and let $\cL_r$ be the set of lines that contain at least two points of $\cP_r$. Then $G$ is transitive on $\cL_r$, each line through a point of $\cP_r$ is in $\cL_r$, and each line in $\cL_r$ contains at least three points of $\cP_r$.\eop
\end{corollary}

An {\em $m$-ovoid} of $\cS$ is a set $M$ of points such that each line of $\cS$ contains exactly $m$ points in $M$ \cite{TSmOv, mOvoTh}. If $M$ is an $m$-ovoid, then
$\vert M \vert = m(st + 1)$, and $|x^\perp\cap M|$ is $1+(t+1)(m-1)$ or $(t+1)m$ according as $x$ is in $M$ or not. 

\begin{corollary}\label{cor_GlinTrMovoid}
The group $G$ is transitive on the set of lines of $\cS$, and each $G$-orbit $\cP_r$ on points is an $m_r$-ovoid for some $m_r\ge 3$.
\end{corollary}
\begin{proof}
Let $\mathscr{R}$ be as in \eqref{eqn_Rdef}. For two distinct primes $r,r'$ in $\mathscr{R}$, we deduce that either $\cL_r=\cL_{r'}$ or $\cL_r\cap \cL_{r'}=\emptyset$ by Corollary \ref{cor_LrProp}. Moreover, $\cL_r=\cL_{r'}$ if and only if there is a line that contains both an $r$-point and an $r'$-point, and we have $\cL=\cup_{r\in\mathscr{R}}\cL_r$. Suppose that there are primes $r,r'\in \mathscr{R}$ such that $\cL_r\cap \cL_{r'}=\emptyset$, i.e., no point of $\cP_r$ is collinear with a point of $\cP_{r'}$.  Take $x\in\cP_r$ and $y\in\cP_{r'}$, and let $z$ be an $r''$-point in $\{x,y\}^\perp$ for some $r''\in\mathscr{R}$. We then have $\cL_{r}=\cL_{r''}=\cL_{r'}$: a contradiction. We have shown that $\cL_r=\cL_{r'}$ for all $r,r'\in \mathscr{R}$, so $\cL_r=\cL$ for  $r\in \mathscr{R}$. It follows  that $G$ is transitive on $\cL$ by Corollary \ref{cor_LrProp}.

Let $\cP_r$ be as in \eqref{eqn_Prdef}, and note that it is a $G$-orbit by Lemma \ref{lem_PrProp} (2). Since $G$ is line-transitive, each line of $\cS$ contains exactly $m_r$ points for some constant $m_r$. By Corollary \ref{cor_LrProp}, we have $m_r\ge 3$. This completes the proof.
\end{proof}

\begin{lemma}\label{lem_subGQ}
The generalized quadrangle $\cS$ does not contain a subquadrangle of order $(s',t)$, where $1<s'<s$.
\end{lemma}
\begin{proof}
Suppose to the contrary that $\cS'=(\cP',\cL')$ is a subquadrangle of order $(s',t)$ of $\cS$.  For two noncollinear points $x,y$ of $\cS'$, the set $\{x,y\}^\perp$ of common neighbors of $x,y$ in $\cS$ are all in $\cP'$, and similarly we have $\{x,y\}^{\perp\!\perp}\subseteq\cP'$. For a nontrivial element $g\in\wE(y)$, we deduce that $x^g$ is in $\cP'$ by Lemma \ref{lem_cover}. Therefore, the restriction of $g$ to $\cS'$ is a nontrivial symmetry about the point $y$. That is, there is a nontrivial symmetry of $\cS'$ about each of its points. Since $\cS$ is a minimal counterexample, we deduce that $\cS'$ is one of $\mW(3,q) $, $\mH(3,q^2)$ and $\mH(4,q^2)$, where $q = r^m$ for a prime $r$ and some $m \in \mathbb{N}_0$. It follows that $t$ is a power of $r$. Since the order of a nontrivial central symmetry of $\cS$ divides $t$, we deduce that $|\wE(x)|\equiv 0\pmod{r}$ for each point $x\in\cP$. By \cite{Ealy} and Theorem \ref{thm_EalyOdd}, $\cS$ cannot be a counterexample to Theorem \ref{Ealygen}. This completes the proof.
\end{proof}

For a non-incident point-line pair $(x,\ell)$, let $\texttt{proj}_\ell(x)$ be the unique point incident with $\ell$ that is collinear with $x$.
\begin{lemma}\label{lem_homFaith}
Let $x,y$ be two noncollinear points such that $G_{[x, y]}\ne \{ \id \}$. Then $G_{[x, y]}$ acts freely on $\ell\setminus\{x,\texttt{proj}_\ell(y)\}$ for each line $\ell$ incident with $x$.
\end{lemma}
\begin{proof}
Take a nontrivial element $g\in G_{[x, y]}$, and assume that it fixes a further point collinear with $x$ besides $\{x,y\}^{\perp}$. By examining the possible substructures of $\cS$, we deduce that the fixed substructure of $g$ must be a subquadrangle of order $(s',t)$ for some integer $s'$ such that $1<s'<s$. This contradicts Lemma \ref{lem_subGQ}, so $g$ fixes no point in $x^\perp\setminus\{x,y\}^\perp$. This completes the proof.
\end{proof}

\begin{lemma}\label{lem_homxyPerp}
For a pair of noncollinear points $x,y$ that are in distinct $G$-orbits, the points in $\{x,y\}^\perp$ are in the same $G$-orbit.
\end{lemma}
\begin{proof}
Let $\ell_0,\ell_1,\ldots,\ell_t$ be the lines through $x$, and let $z_i$ be the unique point on $\ell_i$ that is collinear with $y$ for each $i$. Suppose that $z_0$ is in the $G$-orbit $\cP_r$. By Corollary \ref{cor_GlinTrMovoid}, $m_r=|\ell_i\cap\cP_r|$ is a constant independent of $i$. By Lemma \ref{lem_homology}, there exists a nontrivial homology group $G_{[u, v]}$ for some distinct points $u,v$ in $\{x,y\}^{\perp\!\perp}$. By Lemma \ref{lem_homFaith},  $G_{[u, v]}$ acts freely on $\ell_i\setminus\{x,z_i\}$ for each $i$. Let $n=|G_{[u, v]}|$. We deduce that $m_r\equiv |\cP_r\cap\{x,z_i\}|\pmod{n}$ for each $i$. Since $z_0$ is in $\cP_r$, we have $m_r\equiv |\cP_r\cap\{x\}|+1\pmod{n}$. It follows that
$|\cP_r\cap\{z_i\}|\equiv 1 \pmod{n}$, i.e., $z_i\in\cP_r$ for each $i$. This completes the proof.
\end{proof}

We are now ready to complete the proof of Theorem \ref{Ealygen}. Take the same notation as introduced in the beginning of this section. Fix a prime $r\in \mathscr{R}$. By Corollary \ref{cor_GlinTrMovoid}, $\cP_r$ is an $m_r$-ovoid for a constant $m_r\ge 3$. Let  $x,z$ be two $r$-points on a line $\ell$, and let $\ell'$ be another line incident with $z$. There are $s+1-m_r$ points on $\ell'$ which are not $r$-points, and we have $\{x,y\}^{\perp}\subseteq\cP_r$ for each such point $y$ by Lemma \ref{lem_homxyPerp}. Those $s+1-m_r$ subsets of $\cP_r$ pairwise intersect in $\{z\}$, and each also intersects  $\ell$ in $\{z\}$. It follows that $|x^\perp\cap\cP_r|\ge m_r+(s+1-m_r)t$, i.e., $1+(t+1)(m_r-1)\ge m_r+(s+1-m_r)t$. We deduce that $m_r\ge \frac{s}{2}+1$ for $r\in \mathscr{R}$. Since $\sum_{r\in \mathscr{R}}|\cP_r|=|\cP|$, we deduce that $1+s\ge |\mathscr{R}|(s/2+1)$. This does not hold for $|\mathscr{R}|\ge 2$: a contradiction. This completes the proof of Theorem \ref{Ealygen}. \eop \\

\begin{remark}[The Minimized Ealy Problem] {\rm
If we would consider the combinatorial version of the general Ealy problem, we would be inclined to classify finite generalized quadrangles whose every hyperbolic line is nontrivial --- that is, has size at least $3$. Without any further constraints on the parameters, divisibility conditions on the sizes of the hyperbolic lines, assumptions on the dual hyperbolic lines as well or ``mild'' group-theoretical properties, the problem seems intractable. On the other end of the spectrum, classifying those finite generalized quadrangles with only trivial hyperbolic lines and dual hyperbolic lines might be as difficult, although one might be tempted to conjecture that with the appropriate group-theoretical assumptions, such quadrangles cannot exist with possibly finitely many exceptions. In the infinite case, on the other hand, many examples exist: if one performs a standard free construction on a finite non-closed begin configuration such that we end up with a thick countable free generalized quadrangle, one notices that all hyperbolic and dual hyperbolic lines are trivial. (The essential reason is that in each step, we add a {\em new} incident point-line pair for each non-incident point-line pair $(a,B)$ for which there is no point incident with $B$ which is collinear with $a$ \cite[1.2.13]{vanMa}; this immediately implies that we cannot obtain point and line spans in any way beyond trivial ones.)}
\end{remark}

\section{Applications to locally primitive generalized quadrangles}\label{sec_Kanto}

The classification of flag-transitive generalized polygons is one of the outstanding problems listed in the appendix of \cite{vanMa}. Kantor made the following conjecture in \cite{KantConj}.
\begin{conjecture}If $\cS$ is a finite flag-transitive generalized quadrangle and $\cS$ is not a classical generalized quadrangle, then (up to duality) $\cS$ is the unique generalized quadrangle of order $(3, 5)$ or the generalized quadrangle of order $(15, 17)$ arising from the Lunelli-Sce hyperoval.
\end{conjecture}\label{KanConj}

Let $\cS$ be a finite thick generalized quadrangle of order $(s,t)$ with a flag-transitive automorphism group $G$. Let $\Upgamma$ be the incidence graph of  $\cS$, and let $d(x,y)$ be the distance function on $\Upgamma$ between two vertices $x,y$. For a vertex of $\Upgamma$ and $i\in\mathbb{N}$, we define
\[
 G_x^{[i]}=\{g\in G_x\mid  y^g=y\textup{ for each vertex $y$ with }d(x,y)\le i\}.
\]
By \cite{BamLiSw21}, Conjecture \ref{KanConj} holds if  $G$ is locally $2$-arc-transitive on $\Upgamma$, or equivalently $G$ is transitive on both pairs of collinear points and pairs of concurrent lines by \cite[Lemma 3.2]{localsarc}. It is natural to consider Conjecture \ref{KanConj} under the assumption that $G$ acts locally primitively on $\Upgamma$, i.e., $G_p$ is primitive on the $t+1$ lines through a point $p$ and $G_\ell$ is transitive on the $s+1$ points on a line $\ell$. By combining Theorems 1.1 and 1.2 of \cite{vanBon}, we have the following Thompson-Wielandt-like theorem.

\begin{theorem}\label{ThomWie}
Let $\cS$ be a finite thick generalized quadrangle with an automorphism group $G$, and let $\Upgamma$ be its incidence graph. Suppose that for each vertex $x$, the induced action of $G_x$ on its neighbors in $\Gamma$ is (quasi)primitive. Let $\{x,y\}$ be an edge. Then
\begin{itemize}
  \item[(i)]either $G_x^{[1]}\cap G_{y}^{[1]}$ is an $r$-group for a prime $r$, or (possibly after interchanging $x,y$) $G_x^{[1]}\cap G_{y}^{[1]}=G_x^{[2]}$ and the group $G_y^{[2]}=G_y^{[3]}$ is an $r$-group with $r$ prime;
  \item[(ii)]either there is a prime $r$ such that $F^*(G_{x,y})$, $F^*(G_{x})$ and $F^*(G_{y})$ are $r$-groups, or (possibly after interchanging $x,y$) $G_x^{[1]}\cap G_{y}^{[1]}=G_x^{[2]}$ and $G_y^{[2]}= \{ \id \}$.
\end{itemize}
\end{theorem}

Suppose that $\cS$ is a counterexample to Conjecture \ref{KanConj}, and assume that it has an automorphism group $G$ that is locally primitive on the incidence graph. Since $G$ is flag-transitive and $\cS$ is not a classical generalized quadrangle, we deduce from our first main result and its point-line dual that $G_x^{[2]}=G_y^{[2]} = \{ \id \}$ for each point $x$ and each line $y$. Take an incident point-line pair $(p,\ell)$. By Theorem \ref{ThomWie}, we have the following two possibilities:
\begin{itemize}
  \item[(a)] $G_p^{[1]}\cap G_\ell^{[1]} = \{ \id \}$, or
  \item[(b)] there is a prime $r$ such that $G_p^{[1]}\cap G_\ell^{[1]}$ is a nontrivial $r$-group, and the generalized Fitting subgroups of $G_{p,\ell}$, $G_p$ and $G_\ell$ are all $r$-groups.
\end{itemize}
We observe that $G_p^{[1]}\cap G_\ell^{[1]}$ is a normal subgroup of $G_{p,\ell}$, and it is contained in the Fitting subgroup of the latter group in the case (b).
It would be natural to examine properties of the subgroups $G_p^{[1]}$ and $G_\ell^{[1]}$, which we leave for future work.

\newpage
\section{Appendix: Ealy's arguments}
Let $\mS=(\mP,\mL)$ be a thick generalized quadrangle of order $(s,t)$. Let $E(p)$ be a group of symmetries about a point $p$, and suppose that there is a prime $r$ such that $|E(p)|\equiv 0\pmod{r}$ for each point $p$. Let $G=\la E(p):p\in\mP\ra$. In this appendix, we reproduce Ealy's proofs in \cite{Ealy} for the results that we quote following Theorem \ref{thm_EalyOdd}.
\begin{lemma}\label{lem_ExEycom}
Let $x,y,z$ be three distinct points.
\begin{enumerate}
  \item[(1)] We have $E(x)\cap E(y)= \{ \id \}$, and $[E(x),E(y)]= \{ \id\}$ if and only if $x\sim y$.
  \item[(2)] If $x,y,z$ are collinear, then $E(x)E(y)\cap E(z)= \{ \id \}$.
\end{enumerate}
\end{lemma}
\begin{proof}
For $g\in E(x)\cap E(y)$, we have $g = \id$ by examining its fixed substructure $\cS_g=(\cP_g,\cL_g)$. The second claim in (1) is the point-line dual version of \cite[8.1.3 (i)]{FGQ}. For a nonidentity $g\in E(x)$ and nonidentity $h\in E(y)$, $gh$ is not a symmetry about any point by \cite[8.1.3 (ii)]{FGQ}. The claim (2) then follows from (1). This completes the proof.
\end{proof}

\begin{corollary}\label{cor_IuperpGl}
Let $u$ be a point and $\ell$ be a line of $\cS$.
\begin{enumerate}
  \item[(1)] The group $I(u^\perp)$ is transitive on $u^\perp\setminus\{u\}$, and $I(u^\perp)\unlhd G_u$.
  \item[(2)] The setwise stabilizer $G_\ell$ is $2$-transitive on the points of $\ell$.
\end{enumerate}
\end{corollary}
\begin{proof}
(1) It is clear that $I(u^\perp)\unlhd G_u$. For two noncollinear points $x,y$ in $u^\perp$, there is an element of $I(x,y)$ that maps $x$ to $y$ by Lemma \ref{lem_Omegauv}.  For two collinear points $x,y$ in $u^\perp$, there is a point $w\in u^\perp$ that is not collinear with either of them. We deduce that there is an element $I(\{x,y,z\})$ that maps $x$ to $y$ by applying  Lemma \ref{lem_Omegauv} to the pairs $(x,z)$ and $(z,y)$.

(2) Take any point $x$ on $\ell$. The setwise stabilizer of $\ell$ in $I(x^\perp)$ is transitive on $\ell\setminus\{x\}$ by (1), so $G_{\ell,x}$ is transitive on $\ell\setminus\{x\}$. It holds for all $x\in\ell$, and the claim follows.
\end{proof}

The collinearity graph $\Upgamma$ of the GQ $\cS=(\cP,\cL)$ is the graph whose vertex set is $\cP$ and such that two vertices are adjacent if and only if they are collinear. It is a strongly regular graph with parameters $(v,k,\lambda,\mu)=((s+1)(st+1),s(t+1),s-1,t+1)$, cf.\cite[Proposition 2.2.18]{SRG}. In particular, it is connected and has diameter $2$.  We deduce from Corollary \ref{cor_IuperpGl} that:
\begin{enumerate}
  \item[(a)] $G$ is transitive on the set of collinear point pairs,
  \item[(b)] $G$ is transitive on points, lines and flags respectively.
\end{enumerate}

\begin{lemma}\label{lem_blockOv}
A hypothetical block $B$ for the action of $G$ on $\cP$  is an ovoid of $\cS$ whose size is $st+1$.
\end{lemma}
\begin{proof}
We first show that all points of $B$ are mutually noncollinear, i.e., $B$ is a partial ovoid. Suppose to the contrary that $B$ has two collinear points $u,v$. Then $B=B^g$ for $g\in I(u^\perp)$, and $B$ contains $u^\perp$ by Corollary \ref{cor_IuperpGl}. We deduce that $B$ contains all points of $\cP$: a contradiction. As a corollary, we deduce from the fact $B$ is $E(u)$-invariant that $|B|\equiv 1\pmod{r}$.

{\color{black} Take two points $u,v\in B$ and a point $x\in\{u,v\}^\perp$.} Each line through $x$ contains at most one point of $B$ by the first paragraph. We claim that $|x^\perp\cap B|=t+1$. Suppose to the contrary that there is a line $\ell$ through $x$ that contains no point of $B$. We consider the set $Y = \{\texttt{proj}_\ell(\widetilde{z}) : \widetilde{z} \in B\}$, which contains $x$ and has size at least $2$ by the previous paragraph.
For each $y\in Y$, the group $E(y)$ stabilizes a point of $B$ and so $B=B^y$, and also $E(y)$ fixes $x$. It follows that $B\cap x^\perp$ is $E(y)$-invariant for each $y\in Y$. Letting $E(a)$ act on $B \cap x^{\perp}$, $a \in \{ u, x\}^{\perp} \setminus \{ u, x \}$, we obtain $|B\cap x^\perp|\equiv 1\pmod{r}$. Taking $y\in Y\setminus\{x\}$ and letting $E(y)$ act, we obtain $|B\cap x^\perp|\equiv 0\pmod{r}$: a contradiction.

Consider $B\cap w^\perp$, where $w$ ranges  over points on the line $\ell$ through $u,x$ except for $u$. Since $I(u^\perp)_\ell$ is transitive on $\ell\setminus\{u\}$ and stabilizes $B$, we deduce that they have the same size $t+1$. Also, it is easy to see that they pairwise intersect in $u$. It follows that $|B|\ge 1+st$. Since a partial ovoid has size at most $st+1$, we deduce that $|B|=1+st$.
\end{proof}

\begin{theorem}\label{thm_Gprim}
 The group $G=I(\cP)$ is primitive on $\cP$.
\end{theorem}
\begin{proof}
Let $B$ be a nontrivial block for the action of $G$ on $\cP$. Then $B$ is an ovoid by Lemma \ref{lem_blockOv}.  By the theory of intriguing sets in \cite{TSmOv}, we have $|x^\perp\cap B|=1+t$ for each $x\not\in B$. It follows that $B$ is $E(x)$-invariant for all points $x$, i.e., $B$ is $G$-invariant. We conclude that $B=\cP$: a contradiction. Therefore, $G$ is primitive on the point set $\cP$ as desired.
\end{proof}


\bigskip
\begin{thebibliography}{99}
\bibitem{AlaBur}
S.~H. Alavi and T.~C. Burness, Large subgroups of simple groups, J. Algebra {\bf 421} (2015), 187--233.

\bibitem{AlaMaxExp}
S.~H. Alavi, M. Bayat and A. Daneshkhah, Finite exceptional groups of Lie type and symmetric designs, Discrete Math. {\bf 345} (2022), no.~8, Paper No. 112894, 22 pp..

\bibitem{oddTrans}
M.~Aschbacher, On finite groups generated by odd transpositions I, Math Z. 127 (1972), 45--56; II, III, IV, J. Algebra {\bf 26} (1973), 451--459, 460--478, 479--491.

\bibitem{AscMax}
M.~G.~Aschbacher, On the maximal subgroups of the finite classical groups, Invent. Math. {\bf 76} (1984), no.~3, 469--514.

\bibitem{spor}
J.~Bamberg and J.~P.~Evans, No sporadic almost simple group acts primitively on the points of a generalised quadrangle, Discrete Math. {\bf 344} (2021), no.~4, Paper No. 112291, 17 pp..

\bibitem{Alt}
J.~Bamberg, M.~Giudici , J.~Morris , G.~F.~Royle, P.~Spiga, Generalised quadrangles with a group of automorphisms acting primitively on points and lines, J. Combin. Theory Ser. A {\bf 119} (2012), no.~7, 1479--1499.

\bibitem{TSmOv}J.~Bamberg, M.~Law and T.~Penttila, Tight sets and $m$-ovoids of generalised quadrangles, Combinatorica {\bf 29} (2009), no.~1, 1--17.


\bibitem{BamLiSw18}J.~Bamberg, C.~H.~Li and E.~Swartz, A classification of finite antiflag-transitive generalized quadrangles, Trans. Amer. Math. Soc. {\bf 370} (2018), no.~3, 1551--1601.

\bibitem{BamLiSw21}J.~Bamberg, C.~H.~Li and E.~Swartz, A classification of finite locally 2-transitive generalized quadrangles, Trans. Amer. Math. Soc. {\bf 374} (2021), no.~3, 1535--1578.

\bibitem{AS}
J.~Bamberg, T.~Popiel and C.~E.~Praeger, Simple groups, product actions, and generalized quadrangles, Nagoya Math. J. {\bf 234} (2019), 87--126.





\bibitem{vanBon}J. van~Bon, Thompson-Wielandt-like theorems revisited, Bull. London Math. Soc. {\bf 35} (2003), no.~1, 30--36.

\bibitem{BT}
A.~Borel and J.~Tits, \'El\'ements unipotents et sous-groupes paraboliques de groupes r\'eductifs. I, Invent. Math. {\bf 12} (1971), 95--104.




\bibitem{magma}
W.~Bosma, J.~J.~Cannon and C.~Playoust, The Magma algebra system. I. The user language, J. Symbolic Comput. {\bf 24} (1997), no.~3-4, 235--265.

\bibitem{BHR}
J.~N.~Bray, D.~F.~Holt and C.~M.~Roney-Dougal, {\it The Maximal subgroups of the Low-Dimensional Finite Classical Groups}, London Mathematical Society Lecture Note Series {\bf 407}, Cambridge Univ. Press, Cambridge, 2013.

\bibitem{SRG}
A.~E.~Brouwer and H.~J. Van~Maldeghem, {\it Strongly Regular Graphs}, Encyclopedia of Mathematics and its Applications {\bf 182}, Cambridge Univ. Press, Cambridge, 2022.


\bibitem{BCN}
A.~E.~Brouwer, A.~M~ Cohen and A.~Neumaier, {\it Distance-Regular Graphs}, Ergebnisse der Mathematik und ihrer Grenzgebiete (3) {\bf 18}, Springer, Berlin, 1989.


\bibitem{DisTrans}
F.~Buekenhout and H.~J.~Van~Maldeghem, Finite distance-transitive generalized polygons, Geom. Dedicata {\bf 52} (1994), no.~1, 41--51.


\bibitem{BurGiu}
T.~C.~Burness and M.~Giudici, {\it Classical groups, derangements and primes}, Australian Mathematical Society Lecture Series {\bf 25}, Cambridge Univ. Press, Cambridge, 2016.

\bibitem{CLSS}
A.~M.~Cohen, M.~W.~Liebeck, J.~Saxl, G.~M. Seitz, The local maximal subgroups of exceptional groups of Lie type, finite and algebraic, Proc. London Math. Soc. (3) {\bf 64} (1992), no.~1, 21--48.



\bibitem{Atlas}
J.~H. Conway, R. Curtis, S. Norton, R. Parker, R. Wilson, {\it $\Bbb{ATLAS}$ of finite groups}, Oxford Univ. Press, Eynsham, 1985. Online database: \url{https://brauer.maths.qmul.ac.uk/Atlas/v3/}


\bibitem{Craven1}
D.~A. Craven, The maximal subgroups of the exceptional groups $\mathsf{F}_4(q)$, $\mathsf{E}_6(q)$ and $^2\mathsf{E}_6(q)$ and related almost simple groups, Invent. Math. {\bf 234} (2023), no.~2, 637--719.

\bibitem{Craven2}
D.~A. Craven, On the Maximal Subgroups of $\mathsf{E}_7 (q) $ and Related Almost Simple Groups, arxiv:2201.07081, 2022.



\bibitem{Ealy}
C.~E. Ealy Jr., {\em Generalized Quadrangles and Odd Transpositions}, Ph.D. dissertation, University of Chicago, 1977.


\bibitem{oddOrder}
W.~Feit and J.~G. Thompson, Solvability of groups of odd order, Pacific J. Math. {\bf 13} (1963), 775--1029.


\bibitem{FS1}
P.~Fong and G.~M.~Seitz, Invent. Math. {\bf 21} (1973), 1--57.

\bibitem{FS2}
 P.~Fong and G.~M.~Seitz, Invent. Math. {\bf 24} (1974), 191--239.

\bibitem{localsarc}M. Giudici, C.~H. Li and C.~E. Praeger, Analysing finite locally $s$-arc transitive graphs, Trans. Amer. Math. Soc. {\bf 356} (2004), no.~1, 291--317.

\bibitem{FinGrps}
D.~Gorenstein, {\it Finite Groups}, second edition, Chelsea, New York, 1980.

\bibitem{GorLyons}
D.~Gorenstein and R.~N. Lyons, The local structure of finite groups of characteristic $2$\ type, Mem. Amer. Math. Soc. {\bf 42} (1983), no.~276, {\rm vii}+731 pp..


\bibitem{GorLySol}
D.~Gorenstein, R.~N. Lyons and R. Solomon, {\em The Classification of the Finite Simple Groups}, A.M.S. Surveys and Monographs, vol. {\bf 40}, No. 3, 1997.


\bibitem{Grove}
L.~C. Grove, {\it Classical Groups and Geometric Algebra}, Graduate Studies in Mathematics {\bf 39}, Amer. Math. Soc., Providence, RI, 2002.

\bibitem{slet2}
D.~G. Higman, Partial geometries, generalized quadrangles and strongly regular graphs, in: {\it Atti del Convegno di Geometria Combinatoria e sue Applicazioni (Univ. Perugia, Perugia, 1970)}, pp. 263--293, Univ. Perugia, Perugia.

\bibitem{GGG}
J.~W.~P. Hirschfeld and J.~A. Thas, {\it General Galois geometries}, Oxford Mathematical Monographs Oxford Science Publications, , Oxford Univ. Press, New York, 1991.


\bibitem{Issacs}
I.~M. Isaacs, {\it Character Theory of Finite Groups}, corrected reprint of the 1976 original, Dover, New York, 1994.

\bibitem{KantConj}W.~M. Kantor, Automorphism groups of some generalized quadrangles, in: {\it Advances in Finite Geometries and Designs (Chelwood Gate, 1990)}, 251--256, Oxford Sci. Publ., Oxford Univ. Press, New York.

\bibitem{Kleide8dim}
P. ~B. Kleidman, The maximal subgroups of the finite 8-dimensional orthogonal groups $\mathsf{P\Upomega}_8^+(q)$ and of their automorphism groups, J. Algebra {\bf 110} (1987),173--242.

\bibitem{KL}
P.~B. Kleidman and M.~W. Liebeck, {\it The Subgroup Structure of the Finite Classical Groups}, London Mathematical Society Lecture Note Series {\bf 129}, Cambridge Univ. Press, Cambridge, 1990.

\bibitem{LSS2trans}C.~H. Li, \'A. Seress and S.~J. Song, $s$-arc-transitive graphs and normal subgroups, J. Algebra {\bf 421} (2015), 331--348.

\bibitem{LieSax}M.~W. Liebeck and J. Saxl, On the orders of maximal subgroups of the finite exceptional groups of Lie type, Proc. London Math. Soc. (3) {\bf 55} (1987), no.~2, 299--330.

\bibitem{LSS}M.~W. Liebeck, J. Saxl and G.~M. Seitz, Subgroups of maximal rank in finite exceptional groups of Lie type, Proc. London Math. Soc. (3) {\bf 65} (1992), no.~2, 297--325.


\bibitem{2F4M}G. Malle, The maximal subgroups of ${}^2\mathsf{F}_4(q^2)$, J. Algebra {\bf 139} (1991), no.~1, 52--69.

\bibitem{Muzi}M.~E. Muzychuk, $V$-rings of permutation groups with invariant metric, PhD thesis, Kiev, 1987.

\bibitem{Pass}
D.~S.~Passman, {\em Permutation Groups}, W. A. Benjamin, Inc., New York--Amsterdam, 1968.

\bibitem{tmodrm}
S.~E. Payne, Finite generalized quadrangles: a survey, in: {\it Proceedings of the International Conference on Projective Planes (Washington State Univ., Pullman, Wash., 1973)}, pp. 219--261, Washington State University Press, Pullman, WA.

\bibitem{FGQ}
S.~E. Payne and J.~A. Thas, {\it Finite Generalized Quadrangles}, second edition, EMS Series of Lectures in Mathematics, European Math. Soc., Z\"urich, 2009.

\bibitem{schaffer}
M. Schaffer, Twisted tensor product subgroups of finite classical groups, Comm. Algebra {\bf 27} (1999), no.~10, 5097--5166.

\bibitem{St}
R. Steinberg, {\it Endomorphisms of Linear Algebraic Groups}, Memoirs of the American Mathematical Society, No. {\bf 80}, Amer. Math. Soc., Providence, RI, 1968.

\bibitem{2F42T}
K.~B. Tchakerian, The maximal subgroups of the Tits simple group, C. R. Acad. Bulgare Sci. {\bf 34} (1981), no.~12, 1637.

\bibitem{KTHVM}
K.~Tent and H.~Van Maldeghem, Moufang polygons and irreducible spherical BN-pairs of rank 2, I., Adv. Math. {\bf 174} (2003), 25--265.

\bibitem{mOvoTh}
J.~A. Thas, Interesting pointsets in generalized quadrangles and partial geometries, Linear Algebra Appl. {\bf 114/115} (1989), 103--131.

\bibitem{JATMOUF}
J.~A.~Thas,
Generalized quadrangles of order $(s,s^2)$. IV. Translations, Moufang and Fong-Seitz,
Discrete Math. {\bf 346} (2023), Paper No. 113641, 11 pp.

\bibitem{ACFGQ}
K.~Thas, {\em Automorphisms and Characterizations of Finite Generalized Quadrangles}, in: {\em Generalized Polygons (Book)}, Proceedings of the Academy Contact Forum ``Generalized Polygons'' 20 October, Palace of the Academies, Brussels, Belgium (2001), 111--172.

\bibitem{SFGQ}
K.~Thas, {\em Symmetry in Finite Generalized Quadrangles}, Frontiers in Mathematics {\bf 1}, Birkh\"{a}user, 2004.

\bibitem{KTirred}
K.~Thas, On a conjecture of W. M. Kantor, and strongly irreducible groups acting on generalized quadrangles, Forum Math. {\bf 16} (2004), 671--679.

\bibitem{JTQM}
J.~Tits, Quandrangle de Moufang, I, in: {\em \OE vres/Collected works, Vol. III} (2013),  Heritage of European Mathematics, European Math. Soc., Z\"{u}rich, pp. 472--487.

\bibitem{vanMa}
H.~Van~Maldeghem, {\it Generalized polygons}, Modern Birkh\"auser Classics, Birkh\"auser/Springer Basel AG, Basel, 1998.


\bibitem{walker77}
M. Walker, On the structure of finite collineation groups containing symmetries of generalized quadrangles, Invent. Math. {\bf 40} (1977), no.~3, 245--265.

\bibitem{walker82}
M. Walker, On central root automorphisms of finite generalised hexagons, J. Algebra {\bf 78} (1982), 303--340.

\bibitem{walker83}
M. Walker, On central root automorphisms of finite generalized octagons, European J. Combin. {\bf 4} (1983), 65--86.



\bibitem{2F42W}
R.~A. Wilson, The geometry and maximal subgroups of the simple groups of A. Rudvalis and J. Tits, Proc. London Math. Soc. (3) {\bf 48} (1984), no.~3, 533--563.


\bibitem{wilson}
R.~A. Wilson, {\it The Finite Simple Groups}, Graduate Texts in Mathematics {\bf 251}, Springer, London, 2009.
\end{thebibliography}
\end{document}